\documentclass[12pt, reqno]{article}

\usepackage{amsmath}
\usepackage{amsthm}
\usepackage{amssymb,latexsym}
\usepackage{enumerate}

\usepackage{amsfonts}
\usepackage{lastpage}
\usepackage{fancyhdr}
\usepackage{fancybox}
\usepackage{mathrsfs}

\def\b{\mathbb }
\def\phi{\varphi }
\def\epsilon{\varepsilon}

\swapnumbers

\theoremstyle{plain}

\newtheorem{theorem}{Theorem}[section]
\newtheorem{corollary}[theorem]{Corollary}
\newtheorem{lemma}[theorem]{Lemma}
\newtheorem{proposition}[theorem]{Proposition}
\theoremstyle{definition}
\newtheorem{definition}[theorem]{Definition}
\newtheorem{remark}[theorem]{Remark}
\newtheorem{remarks}[theorem]{Remarks}

\newtheorem{hypergroups}[theorem]{Some facts and notions on commutative
  hypergroups}
\newtheorem{jack}[theorem]{Spherical polynomials}
\newtheorem{Laguerre}[theorem]{Multivariate Laguerre polynomials}
\newtheorem{Bessel}[theorem]{Multivariate Bessel functions}
\numberwithin{equation}{section}

\begin{document}

\title{Multidimensional Heisenberg convolutions and product formulas for
multivariate  Laguerre polynomials}

\author{
Michael Voit\\
Fakult\"at Mathematik, Technische Universit\"at Dortmund\\
          Vogelpothsweg 87\\
          D-44221 Dortmund, Germany\\
e-mail: michael.voit@math.uni-dortmund.de}

\date{}
\maketitle


\renewcommand{\thefootnote}{}

\footnote{2010 \emph{Mathematics Subject Classification}: Primary 43A62;
  Secondary 33C67. 33C52, 43A90, 43A20}

\footnote{\emph{Key words and phrases}: Heisenberg convolution, matrix cones, Weyl chambers,
  multivariate Laguerre polynomials, multivariate Bessel functions,
product formulas, hypergroups, hypergroup characters.  }

\renewcommand{\thefootnote}{\arabic{footnote}}
\setcounter{footnote}{0}


\begin{abstract}
Let $p,q$ positive integers. The groups $U_p(\b C)$ and $U_p(\b C)\times U_q(\b C) $ act 
on the  Heisenberg group $H_{p,q}:=M_{p,q}(\b C)\times \b R$ canonically as
groups of automorphisms where $M_{p,q}(\b C)$ is the vector space of all
complex $p\times q$-matrices. 
The associated orbit spaces may be identified with $\Pi_q\times \b R$ and  $\Xi_q\times \b R$ respectively 
with the cone $\Pi_q$ of positive semidefinite matrices and the Weyl 
chamber $\Xi_q=\{x\in\b R^q:\> x_1\ge\ldots\ge x_q\ge 0\}$. 

In this paper we compute the associated convolutions on $\Pi_q\times \b R$ and
$\Xi_q\times \b R$
 explicitly depending on $p$.
Moreover, we extend these convolutions  by analytic continuation to series of
convolution 
structures for arbitrary parameters $p\ge 2q-1$. 
This leads for $q\ge 2$ to  continuous series of noncommutative hypergroups on  $\Pi_q\times \b R$ and
 commutative hypergroups on $\Xi_q\times \b R$. In the latter case, we describe the dual space in terms 
of multivariate Laguerre and Bessel functions on  $\Pi_q$ and $\Xi_q$. In
particular,
 we give a non-positive product formula for these 
Laguerre functions on $\Xi_q$.  

The paper extends the known case $q=1$ due to Koornwinder, Trimeche, and others as well as the group case with
 integers $p$ due to Faraut, Benson, Jenkins, Ratcliff, and others. Moreover, it is  closely related
 to product formulas for multivariate
Bessel and other  hypergeometric functions of R\"osler.
\end{abstract}

\section{Introduction}
For positive integers $p\ge q$ consider the vector space $M_{p,q}$ of all
$p\times q$ matrices over $\b C$. Consider the associated Heisenberg group
 $H_{p,q}:=M_{p,q}\times \b R$ with the product
$$(x,a)\cdot (y,b)=(x+y, a+b- {\rm Im\>} tr(x^*y))$$
where $tr$ denotes the trace of the $q\times q$ matrix $x^*y$.
Clearly, the unitary groups $K:=U_p:=U_p(\b C)$ and $K:=U_p\times U_q$ act on $H_{p,q}$ via
$$u(x,a):= (ux,a) \quad\quad\text{ and} \quad (u,v)(x,a):= (uxv^*,a)$$
for $u\in U_p$,      $v\in U_q$,  $x\in M_{p,q}$, and
$a\in\b R$
respectively as  groups of automorphisms. 
The associated orbit spaces may be identified with $\Pi_q\times \b R$ and  $\Xi_q\times \b R$ respectively 
with the cone $\Pi_q$ of complex positive semidefinite matrices and the Weyl 
chamber $\Xi_q=\{x\in\b R^q:\> x_1\ge\ldots\ge x_q\ge 0\}$ of type $B$. It is well-known that 
the Banach-$*$- algebras $M_b(H_{p,q}, K)$ of all $K$-invariant bounded
 regular Borel measures with the convolution as multiplication 
 are commutative always for 
 $K:=U_p\times U_q$, and for $K:=U_p$ for $q=1$ only (in which case the cases
  $K:=U_p:=U_p(\b C)$ and $K:=U_p\times U_1$ lead to the same result).
 Moreover, in these Gelfand pair cases, the associated spherical functions are well-known 
 in terms of multivariate Laguerre and Bessel functions;
 we refer to \cite{BJR1}, \cite{BJR2}, \cite{BJRW}, \cite{C}, 
\cite{F}, \cite{Kac}, \cite{W} and references there for this topic. 

In this paper we compute the associated orbit convolutions on $\Pi_q\times \b R$
 and  $\Xi_q\times \b R$ explicitly depending on 
the dimension parameter $p$. This computation is similar to that of R\"osler \cite{R1} 
where the action of  $U_p$ and $U_p\times U_q$  on $M_{p,q}$ is considered
 for the fields $\b R,\b C,\b H$, and where multivariate Bessel functions  appear as spherical functions. 
Moreover, following \cite{R1}, we extend these convolutions 
 by analytic continuation to series of convolution structures
 for arbitrary parameters $p\ge 2q-1$ by using the famous theorem of Carleson. 
This extension leads for $q\ge 2$ to  continuous series of
 noncommutative hypergroups on  $\Pi_q\times \b R$ and
  continuous series of
 commutative hypergroups on $\Xi_q\times \b R$. In the latter case,
 we shall describe the dual spaces in terms 
of multivariate Laguerre and Bessel functions on  $\Pi_q$ and $\Xi_q$.
 Moreover, we determine further data of these hypergroups 
like the Haar measures, the Plancherel measures, automorphisms  and subhypergroups.

The main results will be as follows: 
For the (for $q\ge 2$ noncommutative) hypergroup structures on
 $\Pi_q\times\b R$ we shall derive in Section 2:

\begin{theorem}\label{main-thm-1} Let $q\ge1$ be an integer and  $p\in]2q-1,\infty[$.
Define a convolution of point measures on  $\Pi_q\times\b R$ by
\begin{align}
(&\delta_{(r,a)} *_{p,q} \delta_{(s,b)})(f)=\\ 
&= \kappa_{p,q}\int_{B_q}
f\bigl(\sqrt{r^2 + s^2 + rws + (rws)^*},a+b- {\rm Im\>} tr(rws)\bigr)\notag\\
&\quad\quad\quad\quad\quad
\cdot
\Delta(I_q-w^*w)^{p-2q}\, dw
\notag
\end{align}
for $f\in C_b(\Pi_q\times\b R)$, $r,s\in \Pi_q$, $a,b\in\b R$, where
$$B_q:= \{ w\in M_{q,q}:\> w^*w<I_q, \> \text{i.e.,} \> I_q- w^*w \> \text{is
  positive definite}\},$$ 
$\Delta$ is the determinant of a $\>  q\times q$ matrix,
and $\kappa_{p,q}>0$ is a suitable constant.
Then this formula
establishes by   unique bilinear, weakly continuous extension an associative convolution on
 $M_b(\Pi_q\times \b R)$. More precisely, $(\Pi_q\times \b R, *_{p,q})$ is a
hypergroup in the sense of Jewett (see \cite{BH}, \cite{J}) 
with $(0,0)$ as identity and with the involution
 $\overline{(r,a)}:= (r,-a)$. Moreover,
\[ \omega_{p,q}(f) = \int_{\Pi_q\times \b R}
f(\sqrt{r}, a) \Delta(r)^{p-q} dr \> da\]
defines a left and right Haar measure.
\end{theorem}

For the commutative hypergroup structures on $\Xi_q\times\b R$ we shall derive in Section 3:

\begin{theorem}\label{main-thm-2} 
 Let $q\ge1$ be an integer and $p\in]2q-1,\infty[$. 
Then  $\Xi_q\times\b R$ carries a commutative hypergroup structure with convolution
\begin{align}
(&\delta_{(\xi,a)} \circ_{p,q} \delta_{(\eta,b)})(f)\\ 
&= \kappa_{p,q}\int_{B_q}\int_{U_q}
f\Bigl(\sigma(\sqrt{\xi^2 + u\eta^2 u^* + \xi w u\eta u^*+  u\eta
  u^*w^*\xi}),\notag \\ 
&\quad\quad\quad\quad\quad\quad\quad\quad
a+b- {\rm Im\>} tr(\xi w u\eta u^*)\Bigr)
\cdot
\Delta(I_q-w^*w)^{p-2q}\, du\, dw
\notag
\end{align}
for $f\in C_b(\Xi_q\times\b R)$, $(\xi,a),(\eta,b)\in \Xi_q\times\b R$
 where $du$ means integration with respect to the normalized Haar measure on
 $U_q$ and $\xi\in\Xi$ is identified with the associated diagonal matrix in $\Pi_q$.
The neutral element of this hypergroup is given by $(0,0)\in\Xi_q\times\b R$, and the involution
 by
$\overline{(\xi,a)} := (\xi,-a)$.
Moreover, a Haar measure 
is given by $d\tilde\omega_{p,q}(\xi,a):=  h_{p,q}(\xi)\> d\xi\> da$
 with the Lebesgue density
\begin{equation}
h_{p,q}(\xi):=\prod_{i=1}^q \xi_i^{2p-2q + 1} \prod_{i<j} (\xi_i^2-\xi_j^2)^2.
\end{equation}
\end{theorem}

Moreover, the dual spaces of these commutative hypergroups, i.e., the sets
 of all bounded continuous multiplicative functions, will be 
described precisely as a Heisenberg fan consisting  of multivariate Laguerre
and Bessel functions which were studied in 
\cite{BF},  \cite{F}, \cite{FK}, \cite{He} and many others.
 As already noticed above, this description is well known for the group 
cases with integer $p$ by \cite{BJRW}, \cite{F}.

In Section 4 we shall use the product formula on  $\Xi_q\times\b R$ of Section
3 
in order to derive a product formula 
for the normalized Laguerre functions
$$\tilde\phi_{\bf m}^p(x):= \frac{l_{\bf m}^p(x^2/2)}{l_{\bf m}^p(0)}=
 e^{-(x_1^2+\ldots+x_q^2)/2} \frac{L_{\bf m}^p(x^2)}{L_{\bf m}^p(0)}
\quad\quad (x\in\Xi_q)$$
for $p> 2q-1$ which are introduced, for instance, in \cite{FK}. 
We shall show that for all partitions ${\bf m}$ and all $\xi,\eta\in\Xi_q$,
\begin{align}
\tilde\phi_{\bf m}^p(\xi)\cdot \tilde\phi_{\bf m}^p(\eta)&= 
\kappa_{p,q}\int_{B_q}\int_{U_q}
\tilde\phi_{\bf m}^p\Bigl(\sigma(\sqrt{\xi^2 + u\eta^2 u^* + \xi w u\eta u^*+  u\eta u^*w^*\xi})\Bigr)
\notag \\ &\quad\quad\quad \quad
\cdot e^{-i\cdot {\rm Im\>} tr(\xi w u\eta u^*)}
\Delta(I_q-w^*w)^{p-2q}\, du\, dw.
\end{align}
For $q=1$, this formula was derived by Koornwinder \cite{Ko} who also gave
another version of this product formula using Bessel functions.

We  here notice that on all three levels discussed above also degenerated
product formulas are
 available for the limit case $p=2q-1$. We do not consider the case $p< 2q-1$.

It is a pleasure to thank Margit R\"osler for many valuable
 discussions about their multivariate Bessel convolutions and 
multivariate special functions.

\section{Heisenberg convolutions associated with matrix cones}

For positive integers $p,q$ consider the vector space $M_{p,q}$ of all
$p\times q$ matrices over $\b C$. Consider the associated Heisenberg group
 $H_{p,q}:=M_{p,q}\times \b R$ with the product
$$(x,a)\cdot (y,b)=(x+y, a+b- {\rm Im\>} tr(x^*y))$$
where $tr$ denotes the trace of the $q\times q$ matrix $x^*y$.
Clearly, the unitary group $U_p:=U_p(\b C)$ acts on $H_{p,q}$ via
$$u(x,a):= (ux,a) \quad\quad \text{ for} \quad u\in U_p,\> x\in M_{p,q}, \>
a\in\b R$$
as a group of automorphisms. We regard $U_p$ as a compact subgroup of the associated
semidirect product
$G_{p,q}:= U_p \ltimes H_{p,q}$ in the natural way and consider the double coset
space $G_{p,q}//U_p$ which may be also regarded as the space of all orbits of
the action of $U_p$ on $H_{p,q}$ in the canonical way. Moreover, using
uniqueness of polar decomposition of $p\times q$ matrices, we see immediately
that we may identify this space of orbits also with the space $\Pi_q\times\b
R$ with
$$   \Pi_q:= \{z\in   M_{q,q} : \> z \>\>\text{ Hermitian and positive
  semidefinite}\}$$
via
$$U_p((x,a)) \quad\simeq\quad (|x|,a),  $$         
where $|x|:=\sqrt{x^*x}\in \Pi_q$ stands for the unique positive semidefinite square root of
$x^*x\in \Pi_q$.

Now consider the Banach-$*$-algebra $M_b(G_{p,q}||U_p)$ of all
$U_p$-biinvariant bounded signed Borel measures on $G_{p,q}$ with the usual
convolution of measures as multiplication. If we extend the canonical
projection $P:G_{p,q}\to G_{p,q}//U_p\simeq \Pi_q\times\b R$ to measures by
taking images of measures w.r.t.~$P$, this extension becomes an isometric
isomorphism  between the Banach spaces $M_b(G_{p,q}||U_p)$ and $M_b(\Pi_q\times\b R)$.
We thus may transfer the  Banach-$*$-algebra structure on  $M_b(G_{p,q}||U_p)$
to $M_b(\Pi_q\times\b R)$ by this isomorphism and obtain a probability
preserving, weakly continuous convolution $*_{p,q}$ on  $M_b(\Pi_q\times\b R)$
in this way. The pair $(\Pi_q\times\b R,*_{p,q})$ forms a so-called
hypergroup; for general details on hypergroups the the construction above via
double cosets and orbits we  refer to  \cite{BH} and
\cite{J}.

Clearly, this Heisenberg-type convolution $*_{p,q}$ on (measures on)
$\Pi_q\times\b R$ is commutative iff so is  $M_b(G_{p,q}||U_p)$ , i.e., iff
 $(G_{p,q},U_p)$ is a Gelfand pair. As Gelfand pairs associated with
Heisenberg groups were classified completely 
 (see  \cite{BJR2},  \cite{C}, \cite{Kac},  \cite{W}), it turns out that the 
 convolution $*_{p,q}$ is commutative precisely for $q=1$. Moreover, for
 $q=1$, the convolutions  $*_{p,q}$ on $\Pi_1\times\b R=[0,\infty[\times\b R$ and
 the associated hypergroup structures were investigated by several
authors; see \cite{Ko},  and the monographs \cite{T}, \cite{BH} as well as
references therein.

We next compute the convolution  $*_{p,q}$ for arbitrary positive integers
$q$ under the technical restriction $p\ge 2q$ which will become clear below.
We do this by using the approach for the Gelfand pair $(U_p \ltimes M_{p,q},U_p)$
in   \cite{R1} where the double coset space  $(U_p \ltimes M_{p,q}//U_p)$ is
identified with  $\Pi_q$, and where the same restriction appears. The
computation 
here is only slightly more involved,
and we obtain:

\begin{proposition}\label{computation-convo} Let  $p\ge 2q\ge1$ be integers.
 Then the convolution  $*_{p,q}$
of point measures is given by
\begin{align}\label{convo1}
(&\delta_{(r,a)} *_{p,q} \delta_{(s,b)})(f)\\ 
&= \kappa_{p,q}\int_{B_q}
f\bigl(\sqrt{r^2 + s^2 + rws + (rws)^*},a+b- {\rm Im\>} tr(rws)\bigr)
\notag\\ &\quad\quad\quad\quad
\cdot
\Delta(I_q-w^*w)^{p-2q}\, dw
\notag
\end{align}
for $f\in C_b(\Pi_q\times\b R)$, $r,s\in \Pi_q$, $a,b\in\b R$ with
$I_q\in M_{q,q}$ the identity matrix,
\begin{align}
&B_q:= \{ w\in M_{q,q}:\> w^*w<I_q, \> \text{i.e.,} \> I_q- w^*w \> \text{is
  positive definite}\},\notag\\
&dw \>\>  \text{denoting integration w.r.t.~Lebesgue measure on}\> M_{q,q}, \notag\\
&\Delta\>\> \text{ denoting the determinant of a $\>  q\times q$ matrix, and} \notag\\
&\kappa_{p,q}:=\Biggl(\int_{B_q}\Delta(I_q-w^*w)^{p-2q}\,
dw\Biggr)^{-1}>0.\notag
\end{align}
\end{proposition}

\begin{proof} The canonical projection $\phi:H_{p,q}\to H_{p,q}^{U_p}\simeq
  \Pi_q\times\b R$ from the Heisenberg group onto its orbit space is given
  explicitly by $\phi(x,a)=(|x|, a)$ with $|x|:=\sqrt{x^*x}$. 
Moreover, if we define  the block matrix
\[ \sigma_0 :=  \begin{pmatrix} I_q\\ 0
\end{pmatrix} \in M_{p,q},\]
an ``orbit'' $(r,a)\in \Pi_q\times\b R$ has the representative $( \sigma_0r,a)\in H_{p,q}$. 
By the general definition of the orbit convolution $*_{p,q}$
 (see Section 8.2 of \cite{J} or  \cite{R1})
we have
\begin{align}
(\delta_{(r,a)} *_{p,q} \delta_{(s,b)})(f) &= (\delta_{U_p(\sigma_0r,a)} *_{p,q} \delta_{U_p(\sigma_0s,b)})(f)
\notag\\
 &= \int_{U_p} f\bigl(\phi\bigl((\sigma_0r,a)\cdot u(
(\sigma_0s,b)\bigr)\bigr) \> du\notag\\
 &= \int_{U_p} f\bigl(|\sigma_0r+u\sigma_0s|, a+b- {\rm Im\>} tr(r\sigma_0^*u\sigma_0s)\bigr)
 \> du\end{align}
where $du$ denotes integration w.r.t.~the normalized Haar measure on
$U_p$. Using the definition of the absolute value of a matrix above and
denoting the upper $q\times q$ block of $u$ by  $u_q:=\sigma_0^*u\sigma_0\in
M_{q,q}$, we readily obtain
$$(\delta_{(r,a)} *_{p,q} \delta_{(s,b)})(f)\> =\>
\int_{U_p} f\bigl(\sqrt{r^2+s^2+ru_qs+(ru_qs)^*}, a+b- {\rm Im\>} tr(ru_qs)\bigr)
 \> du.$$
The truncation lemma 2.1 of \cite{R2} now implies the proposition.
\end{proof}

\begin{remarks}
\begin{enumerate}
\item[\rm{(1)}]  The integral in Eq.~(\ref{convo1}) exists precisely for
  exponents $p-2q>-1$ which shows that a formula for $ *_{p,q}$ of the above
  kind is available precisely for $p\ge 2q$. 
\item[\rm{(2)}]  Let  $p\ge 2q\ge1$ be integers,
 and let  $f\in C_b(\Pi_q\times\b R)$, $r,s\in \Pi_q$, $a,b\in\b R$.
Formula~(\ref{convo1}) and a straightforward computation  yield that
\begin{align}
(&\delta_{(s,b)} *_{p,q} \delta_{(r,a)})(f)\> \notag \\
&=\>\kappa_{p,q}\int_{B_q}
f\bigl(\sqrt{r^2 + s^2 + rws + (rws)^*},a+b+ {\rm Im\>} tr(rws)\bigr)
\notag \\ &\quad\quad\quad
\cdot
\Delta(I_q-w^*w)^{p-2q}\, dw.
\notag\end{align}
If one compares this with  Eq.~(\ref{convo1}), the reader can  check directly
the known fact that
 $ *_{p,q}$ is non-commutative precisely for $q\ge 2$. For this, take for
 instance, $a=b=0$, $r=\left(\begin{array}{cc}
 1&0\\0&{\bf 0}
\end{array}\right)$, and  $s=\left(\begin{array}{cc}
{\bf 0} &0\\0&1
\end{array}\right)$ with the zero matrix $ {\bf 0}\in M_{q-1,q-1}$.
\end{enumerate}
\end{remarks}

We next extend the definition of the Heisenberg convolution in
Eq.~(\ref{convo1}) to noninteger exponents $p\in]2q-1,\infty[$ for a fixed
     dimension parameter $q$ by Carlson's theorem on analytic continuation.
For the convenience of the reader we recapitulate this result from \cite{Ti}, p.186:

\begin{theorem}\label{continuation} Let $f(z)$ be holomorphic in a neighbourhood of
$\{z\in \b C:{\rm Re\>} z \geq 0\}$ satisfying $f(z) = O\bigl(e^{c|z|}\bigr)$
on $\,{\rm Re\>}  z \geq 0$ for some $c<\pi$. 
If $f(z)=0$ for all nonnegative integers $z$, then $f$ is identically zero for
${\rm Re\>}  z>0$.
\end{theorem}

This theorem will lead to the following extended convolution:

\begin{theorem}\label{heisenberg-convolution} Let $q\ge1$ be an integer and  $p\in]2q-1,\infty[$.
Define the convolution of point measures on $\Pi_q\times\b R$ by
\begin{align}\label{convo2}
(&\delta_{(r,a)} *_{p,q} \delta_{(s,b)})(f)\\ 
&= \kappa_{p,q}\int_{B_q}
f\bigl(\sqrt{r^2 + s^2 + rws + (rws)^*},a+b- {\rm Im\>} tr(rws)\bigr)
\notag\\&\quad\quad\quad
\cdot
\Delta(I_q-w^*w)^{p-2q}\, dw
\notag
\end{align}
for $f\in C_b(\Pi_q\times\b R)$, $r,s\in \Pi_q$, $a,b\in\b R$  where
$\kappa_{p,q} $, $dw$, $\Delta$ and other data are defined as in Proposition \ref{computation-convo}
above. Then this formula
defines a weakly
continuous convolution of point measures on $ \Pi_q\times \b R$ 
which can be extended uniquely in a bilinear, weakly continuous way to a
probability preserving, weakly continuous, and associative convolution on
 $M_b(\Pi_q\times \b R)$. More precisely, $(\Pi_q\times \b R, *_{p,q})$ is a
hypergroup with $(0,0)$ as identity and with the involution
 $\overline{(r,a)}:= (r,-a)$. 
\end{theorem}

\begin{proof} It is clear from Eq.~(\ref{convo2}) that the mapping
$$(\Pi_q\times\b R)\times(\Pi_q\times\b R)\to M_b(\Pi_q\times\b R), \quad
((r,a), (s,b))\mapsto \delta_{(r,a)} *_{p,q} \delta_{(s,b)} $$
is probability preserving and weakly
continuous. It is now standard (see \cite{J}) to extend  this convolution
uniquely in a bilinear and   weakly
continuous way to a probability preserving convolution on  $M_b(\Pi_q\times \b
R)$.

To prove associativity, it suffices to  consider point measures.
So let $r,s,t\in \Pi_q$, $a,b,c\in\b R$, and $f\in C_b(\Pi_q)$. Then
\begin{align}
&\delta_{(r,a)}*_{p,q}(\delta_{(s,b)}*_{p,q}\delta_{(t,c)})(f) \notag\\
&= 
{\kappa_{p,q}^2}\int_{B_q}\int_{B_q} f\bigl(H(r, a,s,b,t, c;v,w)\bigr)
\notag\\&\quad\quad\quad\quad\quad\quad\quad
\cdot\Delta(I_q-v^*v)^{p-2q}\Delta(I_q-w^*w)^{p-2q}dv dw \,=:I(p) 
\notag\end{align}
with a certain  argument $H$ 
independent of $p$. Similar, 
\[(\delta_{(r,a)}*_{p,q}\delta_{(s,b)})*_{p,q}\delta_{(t,c)}(f) 
=:  I^\prime(p)\]
admits a similar integral representation with some integrand  $H^\prime$ independent of $p$. 
The integrals $I(p)$ and $I^\prime(p)$ are well-defined and holomorphic in 
$\{p\in \b C: {\rm Re\>}  p >2q -1\}$. 
Furthermore, we know from the group cases above that $\,I(p) = I^\prime(p)\,$ 
 for all integers $p\ge 2q$.
As 
\begin{equation}\label{kappaestim}
|\kappa_{p,q}|=O(|p|^{q^2}) \>\>\text{uniformly in $\{p\in \b C:\>{\rm Re\>}  p >
  2q-1\} $ for  $p\to\infty$}
\end{equation}
(see, for example,  Eq.~(3.9) of \cite{R1}), 
we obtain readily that $$f(p):=I(p+2q-1) - I^\prime(p+2q-1)=O(|p|^{2q^2}),$$ and Theorem
\ref{continuation} ensures that $I(p) = I^\prime(p)$  for all $p>2q-1$. Thus
$*_{p,q}$ is associative.

Finally, it is clear by Eq.~(\ref{convo2}) that $\delta_{(0,0)}$ is the
neutral element. Moreover, as the support $supp
(\delta_{(r,a)}*_{p,q}\delta_{(s,b)})$
of our convolution is obviously independent of $p\in]2q-1,\infty[$, all
     further 
hypergroup axioms from \cite{BH} or \cite{J} regarding the support of
convolution products are obvious, as they are valid
for the group cases with integer $p\ge 2q$.
\end{proof}

\begin{remark}\label{support-remark}  The convolution (\ref{convo2})
 obviously satisfies the following 
support formula: For all $(r,a),(s,b)\in\Pi_q\times \b R$,
\begin{align}
supp(&\delta_{(r,a)} *_{p,q} \delta_{(s,b)})\subset
\notag\\&
\subset \{(t,c)\in\Pi_q\times \b R:
\> \|t\|\le \|r\| +\|s \|, \> |c|\le |a|+|b|+ \|r\|\cdot \|s \|\}
\notag
\end{align}
with the Euclidean norm $\|x\|:=\sqrt{tr(x^*x)}$.
\end{remark}

We next collect some properties of the hypergroups  $(\Pi_q\times \b R, *_{p,q})$
for  $p\in]2q-1,\infty[$. We first turn to examples of automorphisms. For this
we first recapitulate that a homeomorphism $T$ on $\Pi_q\times \b R$ is called
a hypergroup automorphism, if for all $(r,a), (s,b)\in \Pi_q\times \b R$,
$$T(\delta_{(r,a)}*\delta_{\overline{(s,b)}})= \delta_{T(r,a)}*\delta_{\overline{T(s,b)}},$$
where the left hand side means the image of the measure under $T$.

\begin{lemma}\label{automorph}
For all $u\in U_q$ and $\lambda>0$, the mappings
$$T_{u,\lambda}(r,a):=(\lambda uru^*, \lambda^2a)$$
are hypergroup automorphisms on   $(\Pi_q\times \b R, *_{p,q})$.
\end{lemma}

\begin{proof}
Eq.~(\ref{convo2}) yields
\begin{align}
(\delta_{T_{u,\lambda}(r,a)} *_{p,q}& \delta_{T_{u,\lambda}(s,b)})(f)\notag\\ 
= \kappa_{p,q}\int_{B_q}&
f\bigl(\lambda\sqrt{u(r^2 + s^2 + ru^*wus + (ru^*wus)^*)u^*},\notag\\ 
&\quad\quad\quad\quad
\lambda^2(a+b- {\rm Im\>} tr(uru^*wusu^*)\bigr)\cdot
\Delta(I_q-w^*w)^{p-2q}\, dw.
\notag
\end{align}
Using $tr(ut)=tr(tu)$, $\sqrt{utu^*}= u\sqrt t u^*$ and the substitution
$v=u^*wu$, we see that this expression is equal to
\begin{align}
&\kappa_{p,q}\int_{B_q}
f\bigl(\lambda u \sqrt{r^2 + s^2 + rvs + (rvs)^*)}u^*,
\notag\\ 
&\quad\quad\quad\quad\quad\quad
\lambda^2(a+b- {\rm Im\>} tr(rvs)\bigr)
\cdot
\Delta(I_q-w^*w)^{p-2q}\, dw\notag\\ 
&=T_{u,\lambda}(\delta_{(r,a)} *_{p,q} \delta_{(s,b)})
\end{align}
as claimed.
\end{proof}

\begin{remark} The Bessel hypergroups on the matrix cones $\Pi_q$ of \cite{R1}
  admit many more hypergroup automorphisms. In fact, a complete classification
  of all automorphisms on these Bessel hypergroups is given in \cite{V3}. Due to the additional term
  ${\rm Im\>} tr(rws)$ in Eq.~(\ref{convo2}), most of these hypergroup
  automorphisms on  $\Pi_q$ cannot be extended to our Heisenberg convolutions.
\end{remark}

 We next turn to the (left) Haar measure which is
     unique up to a multiplicative constant by  \cite{J}:

\begin{proposition}\label{Haar-cone}
A left Haar measure of the hypergroup $(\Pi_q\times \b R, *_{p,q})$ is given by
\[ \omega_{p,q}(f) = \int_{\Pi_q\times \b R}
f(\sqrt{r}, a) \Delta(r)^{p-q} dr \> da\]
for a continuous function  $f\in  C_c(\Pi_q\times \b R)$ with compact support
and the restriction of the Lebesgue measure $dr$ on the vector space of all
Hermitian $q\times q$ matrices. 

Moreover, this left Haar measure is also a right Haar measure.
\end{proposition}

\begin{proof} We first recall that the Heisenberg groups $H_{p,q}$ are
  unimodular with the usual Lebesgue measure $d\lambda$ as Haar
  measure. Therefore, by general results on orbit hypergroups (see e.g. \cite{J}), the image
  $\phi(d\lambda)$ of $d\lambda$ under the canonical projection
  $\phi:H_{p,p}\to \Pi_q\times \b R$ is a left and right Haar measure on  the
  hypergroup $(\Pi_q\times \b R, *_{p,q})$.
Moreover, the computation in Section 3.1 of  \cite{R1} shows that
$$\phi(d\lambda)(r,a)= c_{p,q} \cdot\Delta(r)^{p-q} dr\> da\in M^+(\Pi_q\times \b
R)$$
with a certain known constant $c_{p,q}>0$. This proves the result for integers
$p\ge 2q$.

For the general case we must check that
\begin{align}\label{Haar-condition}
&\int_{\Pi_q}\int_{\b R} (\delta_{(r,a)} *_{p,q} \delta_{(\sqrt s,b)})(f) \>
\Delta(s)^{p-q}ds \> db \notag\\
=&\int_{\Pi_q}\int_{\b R} (\delta_{(\sqrt s,b)} *_{p,q} \delta_{( r,a)})(f) \>
\Delta(s)^{p-q}ds \> db \notag\\
= &\int_{\Pi_q}\int_{\b R} f(\sqrt s,b) \>
\Delta(s)^{p-q}ds \> db   
\end{align}
for all  $f\in C_c(\Pi_q\times \b R)$, $r\in  \Pi_q$, $a\in \b R$ and $p\in\b
C$ with ${\rm Re\>}  p > 2q-1$, where we use Eq.~(\ref{convo2}) also for the
convolution  for complex $p$.
 Clearly, all expressions are analytic in $p$
for fixed $f,r,a,q$. Moreover, by Eq.~(\ref{convo2}), all three expressions
are bounded by
$$C\|f\|_\infty \kappa_{p,q} \cdot M^{{\rm Re}  (p-q)}$$
with some constant $C$ and
\begin{align}
M:= &\sup\{\Delta(s):\> (s,b)\in\Pi_q\times \b R, \> supp(\delta_{(r,a)}
*_{p,q} \delta_{(\sqrt s,b)})\cap supp\>f\ne \emptyset \}
\notag\\
=&\sup\{\Delta(s):\> (s,b)\in  (r,-a)*_{p,q} supp(f)\} .
 \notag\end{align}
Using the estimate (\ref{kappaestim}) for $\kappa_{p,q}$ and the estimate for
the support of convolution products in Remark
\ref{support-remark}, we obtain that the necessary estimate in Carlson's
theorem \ref{continuation} holds whenever $\|r\|$ and the support of $f$ are
contained in a 
sufficiently small neighborhood of $(0,0)$. Therefore, (\ref{Haar-condition}) holds in this case.

Finally, if  $f\in C_c(\Pi_q\times \b R)$ and  $r\in  \Pi_q$ are arbitrary,
then we choose a sufficiently small scaling parameter $\lambda$ such that
$\lambda r$ and the support of $f_\lambda(s,a):= f(\lambda^{-1} s, \lambda^{-2} a)$
are sufficiently small such that  (\ref{Haar-condition}) holds for $\lambda r$
and  $f_\lambda$. As the scaling map $T_{I_q, \lambda}$ is a hypergroup
automorphism, it follows readily that   (\ref{Haar-condition})  for $\lambda r$
and  $f_\lambda$ is equivalent to   (\ref{Haar-condition})  for $ r$
and  $f$. This completes the proof.  
\end{proof}

\begin{remark}\label{abs-cont}
Eq.~(\ref{convo2}) implies that for  $p>2q-1$ and $(r,a),(s,b)\in\Pi_q\times \b R$ with
positive definite matrices $r,s$, the convolution product $\delta_{(r,a)}
*_{p,q} \delta_{(s,b)}$
admits a density w.r.t.~the Lebesgue measure and hence by the preceding
proposition w.r.t.~the Haar measure of the hypergroup $(\Pi_q\times \b R, *_{p,q})$. 

In fact, in the case $p>2q-1$ consider the linear  map 
$$w\mapsto (r^2 +s^2+rws+(rws)^*,{\rm Im\>} tr(rws))$$
from $B_q\subset \b R^{2q^2}$ to $\Pi_q^\circ\times \b R\subset R^{q^2-1}$ which has
a Jacobi matrix with maximal rank $q^2-1$. As the square root mapping on the
interior $\Pi_q^\circ$ of $\Pi_q$ is a diffeomorphism, the claim follows
immediately from the convolution (\ref{convo2}).
\end{remark}

We next turn to the subhypergroups of $(\Pi_q\times \b R, *_{p,q})$. Recapitulate for this that
 a closed set $X\subset\Pi_q\times \b R $ is called a
subhypergroup, if for all $x,y\in X$, we have $\bar x\in X$ and
$\{x\}*\{y\}:=supp(\delta_x*\delta_y)\subset X$.
We next determine all subhypergroups of $(\Pi_q\times \b R, *_{p,q})$. We
begin with examples of subhypergroups.

\begin{proposition}\label{all-subhyper}
Let $p>2q-1$, $k\in\{1,\ldots,q\}$, and $u\in U_q$. Then
$$X_{k,u}:=\left\{
\left(u\left(\begin{array}{cc}
\tilde r&0\\0&0
\end{array}
\right)u^*, a\right): \> \tilde r\in\Pi_k, a\in\b R\right\}$$
 is a subhypergroup of  $(\Pi_q\times \b R, *_{p,q})$, and the mapping
$$(\tilde r,a)\mapsto \left(u\left(\begin{array}{cc}
\tilde r&0\\0&0
\end{array}
\right)u^*, a\right)$$
is a hypergroup isomorphism between the Heisenberg hpergroup
$(\Pi_k\times \b R, *_{p,k})$ and the subhypergroup  $(X_{k,u}, *_{p,q})$.
\end{proposition}

\begin{proof}
The $X_{k,I_q}$ are obviously subhypergroups by 
 Eq.~(\ref{convo2}). Moreover, using the automorphism $T_{u,1}$
 of Lemma \ref{automorph}, we see that the  $X_{k,u}$ are subhypergroups
for arbitrary $u\in U_q$.

In order to check  that the subhypergroup $X_{k,u}$ is isomorphic 
with the hypergroup
  $(\Pi_k\times \b R, *_{p,k})$, we may  assume $u=I_q$ without
loss of generality. Here we first consider the group cases with integer $p\ge
2q$.
 Here, the inverse image of
 $X_{k,u}$ under the canonical projection $\phi:H_{p,q}\to \Pi_q\times  \b R$ 
 is given by the subgroup
$\left\{\left(\left(\begin{array}{c}
x\\0
\end{array}\right), a\right), \> x\in M_{p,k}, a\in\b R\right\}$ of $H_{p,q}$ 
which is isomorphic with $H_{p,k}$ and preserved by the action of $U_p$.
Thus, the preceding construction of the orbit hypergroup structures implies
that
   $(X_{k,u}, *_{p,q})$ is isomorphic with 
  $(\Pi_k\times \b R, *_{p,k})$ as claimed in this case.
Therefore, for  integers $p\ge 2q$ and all $f\in C_b(\Pi_q\times \b R)$ and $(r,a),(s,b)\in
\Pi_k\times \b R$,
$$
\delta_{\left(\left(\begin{array}{cc}
 r&0\\0&0
\end{array}\right),a\right)} *_{p,q} \delta_{\left(\left(\begin{array}{cc}
 s&0\\0&0
\end{array}\right),b\right)}(f) =
(\delta_{(r,a)} *_{p,k} \delta_{(s,b)})(f_k)$$
with $f_k(r,a):=f\left(\left(\begin{array}{cc}
 r&0\\0&0
\end{array}\right),a\right)$.
If we use the definitions of these convolutions in Theorem
 \ref{heisenberg-convolution} for arbitrary $p$,
 analytic continuation via
Carlson's theorem yields in the same way as in the proof of
  Theorem \ref{heisenberg-convolution} that this equation holds 
for all $p>2q-1$. This completes the proof.
\end{proof}

\begin{remark}
It follows immediately from Eq.~(\ref{convo2}), that $X_0:=\{0\}\times\b R$ is
a normal subgroup 
of  $(\Pi_q\times \b R, *_{p,q})$
 isomorphic with $(\b R,+)$. We now may consider the associated quotient
 hypergroup
 $ (\Pi_q\times \b R, *_{p,q})/X_0$ which can 
be identified with $\Pi_q$ obviously in a topological way. Using the
definition of 
the quotient convolution (see e.g. \cite{V2})
as well as Eq.~(\ref{convo2}), this quotient  convolution on  $\Pi_q$  is given by
$$ (\delta_{r} * \delta_{s})(f)
= \kappa_{p,q}\int_{B_q}
f\bigl(\sqrt{r^2 + s^2 + rws + (rws)^*}\bigr)
\cdot
\Delta(I_q-w^*w)^{p-2q}\, dw .$$
In other words, $ (\Pi_q\times \b R, *_{p,q})/X_0$ is isomorphic
 with the Bessel hypergroup structure on the cone  $\Pi_q$
of \cite{R1} with index $p$.
\end{remark}

\begin{lemma}\label{subgroup1}
Let $p> 2q-1$. Let $X$ be a subhypergroup of  $(\Pi_q\times \b R, *_{p,q})$ 
which is not contained in the subgroup $X_0$. Then $X_0\subset X$. 
\end{lemma}

\begin{proof} Consider a subhypergroup $X\not\subset X_0$. Thus there exist
  $r\in\Pi_q\setminus\{0\}$ and $a\in\b R$ with $(r,a)\in X$. If we restrict
  the integration in Eq.~(\ref{convo2}) to matrices $w_c=(-1/2+ci)\cdot I_q\in B_q$
  with $c\in[-\sqrt 3/2,\sqrt 3/2]$, we conclude from (\ref{convo2}) and
  $\sqrt{2r^2+ rw_cr+rw^*_cr}=r$ that 
$$\{r\}\times [- \sqrt 3 \cdot tr(r^2)/2, \sqrt 3 \cdot tr(r^2)/2]\quad
\subset \quad \{(r,a)\}*_{p,q} \{(r,-a)\}\quad
\subset \quad X.$$
Therefore, by (\ref{convo2}), there exists $\epsilon>0$ such
 that for all $x\in[-\epsilon, +\epsilon]$ we have
 $(0,x)\in \{(r,x)\}*_{p,q} \{(r,-x)\} \subset X$. As $X_0$ is a subgroup
isomorphic with $(\b R,+)$, it follows that $X_0\subset X$.
\end{proof}

\begin{proposition} Let $X$ be a subhypergroup of  $(\Pi_q\times \b R, *_{p,q})$.
 Then $X$ is a subgroup of $X_0$ or $X$ is equal to one of
 the subhypergroups $X_{k,u}$ of Proposition \ref{all-subhyper}.
\end{proposition}

\begin{proof}
Let $X$ be a subhypergroup which is not  contained in $X_0$. Then $X_0\subset X$ by Lemma \ref{subgroup1}, 
and we may consider the quotient subhypergroup $X/X_0$ in the quotient hypergroup $(\Pi_q\times \b R)/X_0 $
 which is isomorphic with the Bessel hypergroup of \cite{R1} on the cone $\Pi_q$ with parameter $p$.
On the other hand, all subhypergroups of the  Bessel hypergroup structures 
 on the $\Pi_q$ were classified in Proposition 4.6 of \cite{V3}. As  $X_0\subset X$, 
this classification leads immediately to the classification above. 
\end{proof}

\begin{remark}\label{spezail1}
Let
$$B:=\{y\in\b C^q:\> \|y\|_2<1\} \quad\quad{\rm and} \quad\quad
S:=\{y\in\b C^q:\> \|y\|_2=1\}.$$
By Lemma 3.6 and Corollary 3.7 of \cite{R1}, the mapping $P:B^q\to B_q$ from
the direct product $B^q$ to the ball $B_q$
with
\begin{equation}\label{trafo-P}
 P(y_1, \ldots, y_q):= \begin{pmatrix}y_1\\y_2(I_q-y_1^*y_1)^{1/2}\\ 
\vdots\\
  y_q(I_q-y_{q-1}^*y_{q-1})^{1/2}\cdots (I_q-y_{1}^*y_{1})^{1/2}\end{pmatrix}
\end{equation}
establishes a diffeomorphism such that the image of the measure 
$$\Delta(I_q-w^*w)^{p-2q}dw$$
under  $P^{-1}$ is given by $\,\prod_{j=1}^{q}(1-\|y_j\|_2^2)^{p-q-j}dy_1\ldots
dy_q$. Therefore, Eq.~(\ref{convo2}) may be written as
\begin{align}\label{convo3}
(&\delta_{(r,a)} *_{p,q} \delta_{(s,b)})(f)\\ 
&= \kappa_{p,q}\int_{B^q}
f\bigl(\sqrt{r^2 + s^2 + rP(y)s + sP(y)^*r},a+b- {\rm Im\>} tr(rP(y)s)\bigr)\notag\\ 
&\quad\quad\quad
\cdot\prod_{j=1}^{q}(1-\|y_j\|_2^2)^{p-q-j}dy_1\ldots
dy_q
\notag
\end{align}
for $p> 2q-1$. Moreover,  for $p\to 2q-1$, this
convolution product converges weakly to the probability measure 
$ \delta_{(r,a)} *_{2q-1,q} \delta_{(s,b)}\in M^1(\Pi_q\times\b R)$ with
\begin{align}\label{convo4}
(&\delta_{(r,a)} *_{2q-1,q} \delta_{(s,b)})(f)\\ 
&= \kappa_{2q-1,q}\int_{B^{q-1}}\int_S
f\bigl(\sqrt{r^2 + s^2 + rP(y)s + sP(y)^*r},a+b- {\rm Im\>} tr(rP(y)s)\bigr)\notag\\ 
&\quad\quad\quad
\cdot\prod_{j=1}^{q-1}(1-\|y_j\|_2^2)^{p-q-j}dy_1\ldots
dy_{q-1}\> d\sigma(y_q)
\notag
\end{align}
where $\sigma\in M^1(S)$ is the uniform distribution on $S$ and
$\kappa_{2q-1,q}>0$ a suitable normalization constant.

This convolution is obviously weakly continuous and can be extended to an
associative, weakly continuous, and probability preserving convolution on
$M_b(\Pi_q\times \b R)$ by Theorem \ref{heisenberg-convolution} and taking the
limit above. Moreover, all further hypergroup axioms may be also checked
readily for (\ref{convo4}). Finally, the measure $\omega_{2q-1,q}$ defined as in Proposition
\ref{Haar-cone} is a 
 Haar measure of this hypergroup $(\Pi_q\times \b R, *_{2q-1,q})$, the
 mappings $T_{u,\lambda}(r,a):=(\lambda uru^*, \lambda^2a)$ are also
 automorphisms here
 as in Lemma \ref{automorph}, and the subsets $X_{k,u}\subset \Pi_q\times \b R$
defined as in Proposition  \ref{all-subhyper} are again subhypergroups.
\end{remark}

\section{Heisenberg-type convolutions associated with Weyl chambers of type B}

In this section we consider the group $U_q$ which acts by Lemma \ref{automorph} 
as a compact group $\{T_{u,1}: \> u\in U_q\}$ of automorphisms on 
the Heisenberg hypergroups $(\Pi_q\times \b R, *_{p,q})$. 
As the orbits of the action of $U_q$ on $\Pi_q$ by conjugation are 
described by the ordered eigenvalues $\xi_1\ge \xi_2\ge\cdots\ge \xi_q\ge0$ of
a matrix in
  $\Pi_q$, we may identify
the space of all $U_q$-orbits of  $(\Pi_q\times \b R, *_{p,q})$ with the set 
$\Xi_q\times \b R$ where 
$$\Xi_q := \{\xi = (\xi_1, \ldots \xi_q)\in \b R^q: 
\xi_1\geq\ldots\geq \xi_q\geq 0\}.$$
The set $\Xi_q$ is a closed Weyl chamber of the hyperoctahedral group $B_q = 
S_q\ltimes \b Z_2^q$ which acts on $\b R^q$ by permutations of the basis
vectors and sign changes.  In this section we  show how the 
convolutions
$*_{p,q}$ on $\Pi_q\times \b R$ for $p\ge 2q-1$ lead to
  orbit hypergroup convolutions $\circ_{p,q}$ on $\Xi_q\times \b R$ by using
  methods of
\cite{J} or \cite{R2}. In contrast to the hypergroups $(\Pi_q\times \b
R,*_{p,q})$, 
 the hypergroups $(\Xi_q\times \b R,\circ_{p,q})$ are always commutative. 
 We shall identify the characters of these hypergroups in terms of 
with multivariate Bessel and Laguerre functions
associated with the root system  $B_q$.

Let us go into the details. Let $q\ge1$ be an integer and $p\in[2q-1,\infty[$.  In the 
 situation described above, the  mapping
$\Pi_q \to \Xi_q\,, \quad r\mapsto \sigma(r)$,
 which assigns to each matrix $r$ its ordered spectrum $\sigma(r)$, 
is continuous, surjective and open w.r.t.~the standard topologies on both sets.
Therefore the orbit space $(\Pi_q \times\b R)^{U_q}$
 (equipped with the quotient topology) may be identified with $ \Xi_q\times\b R$ also 
in a topological way. We now identify both spaces in the
 obvious way and consider the continuous, surjective and open
mapping
$$\Phi:\Pi_q \times\b R\to  \Xi_q\times\b R, \quad (r,a)\mapsto (\sigma(r),a)$$
which corresponds to the orbit map above. This mapping is a orbital mapping
 from the hypergroup $(:\Pi_q \times\b R, *_{p,q})$ onto 
 $\Xi_q\times\b R$ in the sense of Section 13 of \cite{J}, and it follows
 readily from Section 13 of \cite{J} that 
 $\Xi_q\times\b R$ carries a corresponding orbit hypergroup convolution $\circ_{p,q}$ as follows:
For $a,b\in\b R$ and $\xi,\eta\in\Xi_q$ we choose
 representatives $x,y\in\Pi_q$ with $\sigma(x)=\xi$ and  $\sigma(y)=\eta$ and put
\begin{equation}\label{convo-w1}
\delta_{(\xi,a)}\circ_{p,q}\delta_{(\eta,b)}:=\Phi(\delta_{(x,a)}*_{p,q}\delta_{(y,b)}).
\end{equation}
The properties of this hypergroup convolution can now be derived  
similar to Section 4 of \cite{R1}. In particular, we can write down 
the convolution (\ref{convo-w1}) 
 explicitly.
For this, we denote the normalized Haar measure on $U_q$ by $du$,  and $\xi\in
\Xi_q$ will always be  identified with the diagonal
 matrix $\text{diag}(\xi_1,\ldots, \xi_q)\in \Pi_q$ without mentioning.

\begin{theorem}\label{main3}  Let $q\ge1$ be an integer and $p\in]2q-1,\infty[$. 
Then  $\Xi_q\times\b R$ carries a commutative hypergroup structure with the convolution
\begin{align}\label{convo-w2}
(&\delta_{(\xi,a)} \circ_{p,q} \delta_{(\eta,b)})(f)\\ 
&= \kappa_{p,q}\int_{B_q}\int_{U_q}
f\Bigl(\sigma(\sqrt{\xi^2 + u\eta^2 u^* + \xi w u\eta u^*+  u\eta u^*w^*\xi}),
\notag \\ 
&\quad\quad\quad\quad\quad\quad\quad\quad
a+b- {\rm Im\>} tr(\xi w u\eta u^*)\Bigr)
\cdot
\Delta(I_q-w^*w)^{p-2q}\, du\, dw
\notag
\end{align}
for $f\in C_b(\Xi_q\times\b R)$, $(\xi,a),(\eta,b)\in \Xi_q\times\b R$.
The neutral element is given by $(0,0)\in\Xi_q\times\b R$, and the involution
 by
$\overline{(\xi,a)} := (\xi,-a)$.

Moreover, a Haar measure on $(\Xi_q\times\b R, \circ_{p,q} )$ 
is given by $$d\tilde\omega_{p,q}(\xi,a):=  h_{p,q}(\xi)\> d\xi\> da$$
 with the Lebesgue density
\begin{equation}\label{haar-density}
h_{p,q}(\xi):=\prod_{i=1}^q \xi_i^{2p-2q + 1} \prod_{i<j} (\xi_i^2-\xi_j^2)^2.
\end{equation}
\end{theorem}

\noindent 

\begin{proof} In view of Section 13 of  \cite{J} and Section 4 of \cite{R1} on orbit hypergroups,
we only have to check the commutativity of  $\circ_{p,q} $ as well
 as the statement about the Haar measure.

We first turn to the commutativity. We first observe that for integers
$p>2q-1$
 by its  construction, the hypergroup $(\Xi_q\times\b R, \circ_{p,q}) $ 
is  isomorphic with the orbit hypergroup which appears when the group 
$U_p\times U_q$ acts on the Heisenberg group $H_{p,q}$ by
$(u,v)(x,a):=(uxv^*, a)$ for $u\in U_p$, $v\in U_q$, $x\in M_{p,q}$ 
and $a\in \b R$. Moreover, it is well known that 
$$((U_p\times U_q) \ltimes H_{p,q}, U_p\times U_q)$$ is a Gelfand pair; see 
 \cite{BJR1},  \cite{C}, \cite{F}, \cite{Kac}. Therefore, 
 $\circ_{p,q} $ is commutative for integers $p\ge 2q$.
 The general case can now be proved by analytic continuation using Carlson's theorem \ref{continuation} 
in the same way as in the proof of Theorem \ref{heisenberg-convolution}. We
omit the details.

On the other hand, we may check commutativity also directly. In fact, let 
 $a,b\in\b R$ and $\xi,\eta\in\Xi_q$.
 We also regard $\xi,\eta$ as real diagonal matrices as described above. We
 obtain from invariance of  spectrum and trace by
conjugations that
\begin{align}
(&\delta_{(\eta,b)} \circ_{p,q} \delta_{(\xi,a)})(f)
\notag\\ 
&= \kappa_{p,q}\int_{B_q}\int_{U_q}
f\Bigl(\sigma(\sqrt{u^*\eta^2u + \xi^2  + u^*\eta w u\xi +  (u^*\eta w
  u\xi)^*}),\notag \\ 
&\quad\quad\quad\quad\quad\quad\quad\quad\quad
a+b- {\rm Im\>} tr(u^*\eta w u\xi)\Bigr)
\cdot
\Delta(I_q-w^*w)^{p-2q}\, du\, dw.
\notag
\end{align}
Substitution $w\mapsto \bar w$ as well as $dw=d\bar w$, 
 $\Delta(I_q-\bar w^*\bar w)=\Delta(I_q-w^*w)>0$, $\sigma(x^T)=\sigma(x)$,
$tr(x^T)=tr(x)$, $\bar\xi=\xi$, and $\bar\eta=\eta$ imply that
this expression is equal to
\begin{align}
&\kappa_{p,q}\int_{B_q}\int_{U_q}
f\Bigl(\sigma(\sqrt{u^T \eta^2\bar u + \xi^2  + u^T\eta w \bar u\xi +  \xi u^T
  w^* \eta \bar u}),\notag \\ 
&\quad\quad\quad\quad\quad\quad\quad\quad
a+b- {\rm Im\>} tr(\xi u^T  w^* \eta \bar u)\Bigr)
\cdot
\Delta(I_q-w^*w)^{p-2q}\, du\, dw.
\notag
\end{align}
Using the substitution $u\mapsto u^T$, which preserves the Haar measure on $U_q$, as well
 as the substitution $w\mapsto u^*w^*u$, which preserves 
the Lebesgue measure on $B_q$, we obtain that the expression above
 is equal to the right hand side of (\ref{convo-w2}).
 This completes the direct proof of commutativity.

We finally turn to the Haar measure. By Section 13 of \cite{J},
 the Haar measure $\tilde\omega_{p,q}\in M^+(\Xi_q\times\b R)$ is just given as
the image of the Haar measure $\omega_{p,q}\in M^+(\Pi_q\times\b R)$ under the
projection $\Phi$. As here the second component $\b R$ is not involved,
 the computation of this image measure can be carried out in the same way as 
in the corresponding proof for the matrix Bessel hypergroups
 in Theorem 4.1 of \cite{R1}. We therefore omit the details.
\end{proof}

\begin{remarks}\label{spezial2}
\begin{enumerate}
\item[\rm{(1)}] For $p=2q-1$ the convolution $*_{2q-1,q}$ on $\Pi_q\times\b R$
  introduced in Section \ref{spezail1} can be also transfered 
to a commutative hypergroup convolution  $\circ_{2q-1,q}$ on $\Xi_q\times\b R$
in the same way as above. We here omit details.
\item[\rm{(2)}] As already mentioned in the preceding theorem, 
the hypergroups  $(\Xi_q\times\b R,\circ_{p,q})$ 
are orbit hypergroups associated with the action of $U_p\times U_q$ 
on the Heisenberg group $ H_{p,q}$ for integers $p\ge 2q$.
 Clearly, one may also form the
associated  orbit hypergroup structures $(\Xi_q\times\b R,\circ_{p,q})$ for all integers $p\ge q$, 
where then the corresponding convolution for $p=q,q+1,\ldots,2q-1$
is degenerated and no longer given by (\ref{convo-w2}).

\item[\rm{(3)}] It is clear by  the convolution (\ref{convo-w2}) that
  $G:=\{0\}\times\b R$ is a subgroup of  $(\Xi_q\times\b R,\circ_{p,q})$
  isomorphic with $(\b R,+)$. We thus may form the quotient hypergroup
$$(\Xi_q\times\b R)/G:=\{G\cdot(\xi,a)=(\xi,\b R):\> (\xi,a)\in\Xi_q\times\b
R\} \simeq \Xi_q.$$
Using this natural identification as well as the canonical
 projection $\Psi:\Xi_q\times\b R \to\Xi_q$, the quotient convolution is defined by
 $$\delta_\xi\bullet_{p,q}\delta_\eta:= \Psi(\delta_{(\xi,0)} \circ_{p,q}
 \delta_{(\eta,0)}),$$
i.e.,\begin{align}\label{convo-bessel}
(&\delta_{\xi} \circ_{p,q} \delta_{\eta})(f)\\ 
&= \kappa_{p,q}\int_{B_q}\int_{U_q}
f\bigl(\sigma(\sqrt{\xi^2 + u\eta^2 u^* + \xi w u\eta u^*+  u\eta u^*w^*\xi})\bigr)
\notag\\ &\quad\quad\quad\quad\quad\quad
\cdot
\Delta(I_q-w^*w)^{p-2q}\, du\, dw
\notag
\end{align}
for $f\in C_b(\Xi_q)$, $\xi,\eta\in \Xi_q$. In other words, the quotient
hypergroup $((\Xi_q\times\b R)/G, \bullet_{p,q})$ is precisely the
Bessel-hypergroup on the Weyl-chamber $\Xi_q$ as studied in Section 4 of
\cite{R1} for the field $\b C$, i.e., the parameter $d=2$ there.
\end{enumerate}
\end{remarks}

We next turn to the characters of the commutative hypergroups
$(\Xi_q\times\b R,\circ_{p,q})$. 
For this we first recapitulate some basic notions and facts about commutative
hypergroups mainly from \cite{J} and \cite{BH}.

\begin{hypergroups}\label{facts-comm-hyper}
Let $(X,*)$ be a commutative hypergroup. Then there is a Haar measure
$\omega\in M^+(X)$ which is unique up to a multiplicative constant. 
We introduce the dual space
$$\hat X:=\{\alpha\in C_b(X):\quad  \delta_x*\delta_{\bar
  y}(\alpha)=\alpha(x)\overline{\alpha(y)}\quad\text{ for all } x,y\in X\}$$
and the space of all multiplicative functions
$$\chi_b(X):=\{\alpha\in C_b(X):\quad  \delta_x*\delta_{y}
(\alpha)=\alpha(x){\alpha(y)}\quad\text{ for all } x,y\in X\},$$
and equip both with the topology of locally uniform convergence. 
Both spaces are locally compact, and for a Gelfand pair $(G,K)$, the space of
  spherical functions corresponds to the space $\chi_b(X)$ for the double
  coset hypergroup $(G//K,*)$. The elements
of $\hat X$ are called characters.

We define the Fourier transform $.^\wedge: L^1(X,\omega)\to C_0(\hat X)$ with
$$\hat f(\alpha):=\int_X \overline{\alpha(x)} \cdot f(x)\> d\omega(x).$$
Then there exists a unique Plancherel measure $\pi\in M^+(\hat X)$ such that
the Fourier transform becomes an $L^2$-isometry, i.e., for all $f\in
L^1(X,\omega)\cap L^2(X,\omega)$ we have $\int_X |f|^2\> d\omega = \int_{\hat X}
|\hat f|^2\> d\pi$, and the Fourier transform can be extended to an isometric
isomorphism between $L^2(X,\omega)$ and $L^2(\hat X,\pi)$.

Different to the case of abelian groups, it may occur that $supp\> \pi\ne \hat
X\ne\chi_b(X) $. This is the case for instance for Gelfand pairs associated with
noncompact semisimple Lie groups. On the other hand, there is a growth criterion
in hypergroup theory which ensures $ supp\> \pi=\hat X =\chi_b(X)$. 
To explain this, take a compact set $A\subset X$ and define recursively the
sets $A^{(n)}$ by $A^{(1)}=A$ and
$A^{(n+1)}=A^{(n)}*A^{(1)}=\bigcup_{x\in A^{(n)}, y\in A^{(1)}}
supp(\delta_x*\delta_y) $.
We say that $(X,*)$ has subexponential growth if for all compact sets
$A\subset X$ and all $c>1$ we have $\omega( A^{(n)})=o(c^n)$ for $n\to\infty$.
It was proved in \cite{Vog} and \cite{V1} that for each commutative hypergroup $(X,*)$
with subexponential growth,  $ supp\> \pi=\hat X =\chi_b(X)$.
\end{hypergroups}

We now return to the hypergroups $(\Xi_q\times\b R,\circ_{p,q})$. As the
Heisenberg groups have polynomial growth, the following result is not surprising:

\begin{lemma} The hypergroups $(\Xi_q\times\b R,\circ_{p,q})$ have
  subexponential growth for  $p\ge 2q-1$.
\end{lemma}

\begin{proof} We see from (\ref{convo-w2}) that for all $(\xi,a),(\eta,b)\in
  \Xi_q\times\b R$ and $(\tau,c)\in supp(\delta_{(\xi,a)} \circ_{p,q}
  \delta_{(\eta,b)})$ the first (and thus largest) components of the vectors $\tau,\xi,\eta$ satisfy
 $\tau_1\le\xi_1+\eta_1$ and $|c|\le |a|+|b|+
  \xi_1\eta_1$.
Now let $C\subset \Xi_q\times\b R$ be compact. Choose $d>0$ such that $\xi_1\le d$ and
  $|a|\le d$ for all $(\xi,a)\in C$. A simple induction shows that then for
  all $n\in\b N$ and all $(\tau,c)\in C^{(n)}$ we have $0\le\tau_q\le
  \ldots\le \tau_1\le nd$ and $|c|\le nd+\frac{n(n-1)}{2} \cdot d^2$.
As the Haar measure $\tilde\omega_{p,q}$ has a polynomially growing Lebesgue
  density by Theorem \ref{main3}, the assertion is clear. 
\end{proof}

By the results of Section \ref{facts-comm-hyper} we obtain:

\begin{corollary}\label{mult} The hypergroups
  $(\Xi_q\times\b R,\circ_{p,q})$  satisfy $ supp\> \pi=\hat X =\chi_b(X)$.
\end{corollary}

Consider the canonical projection $\Phi:\Pi_q\times\b R \to
\Xi_q\times\b R$ as in the beginning of this section. This mapping is an
orbital morphism in the sense of \cite{J}, and we conclude from   \cite{J}:

\begin{corollary}\label{lifting} 
Let $p\ge2q-1$. For each character $\alpha$ of
$(\Xi_q\times\b R,\circ_{p,q})$, the function $\alpha\circ\Phi\in
C_b(\Pi_q\times\b R)$ is positive definite on the hypergroup $(\Pi_q\times\b R,*_{p,q})$.
\end{corollary}

We next introduce a set $\Sigma_{p,q}$ of characters of the hypergroups  $(\Xi_q\times\b
R,\circ_{p,q})$. Later on we shall see that this set in fact consists of all
characters. This set $\Sigma_{p,q}$ consists of two disjoint sets
$\Sigma_{p,q}^1$ and $\Sigma_{p,q}^2$  of functions where these functions are described in
terms  of multivariate Laguerre and Bessel
functions respectively  as discussed in \cite{FK}. This is not surprising,
as this connection is well-known for $q=1$ (see the product formula in
\cite{Ko}, \cite{T} and references cited there) as
well as
  for the group cases with integers  $p, q\ge1$; see \cite{F} and references there.
 Before going into details, we  collect
some notions and facts from \cite{BF}, \cite{F}, \cite{FK}, and \cite{Kan}.
 We start with some basic notions on multivariate special functions:

\begin{jack}\label{Jack}
Let ${\bf m}=(m_1,\ldots,m_q)$  be a partition of length $q$ with integers $m_1\ge
m_2\ge\ldots\ge m_q\ge0$. 
We define its length $|{\bf m}|:=m_1+\ldots +m_q$, 
the generalized Pochhammer symbol 
\begin{equation}
 (x)_{\bf m} \,=\,\prod_{j=1}^q \bigl(x-j+1\bigr)_{m_j}
\end{equation}
for  $x\in\b R$ (note that we here always use $d=2$ in the notion of
 \cite{FK}), as well as the
 dimension constant
$$d_{\bf m}:=\frac{(p)_{\bf m}(q)_{\bf m}}{h({\bf m})^2}$$
where $h({\bf m})$ is the product of the hook lengths of ${\bf m}$; see pp. 237
of \cite{F} and p.~66 of \cite{M2}. Moreover, for partitions ${\bf m}$
we define the  spherical polynomials
$$ \Phi_{\bf m} (x) = \int_{U_q} \Delta_{\bf m}(uxu^{-1})du
\quad\text{for}\quad  x\in M_{q,q}$$
where $du$ is the normalized Haar measure of $U_q$,
 $\Delta_{\bf m}$ is the power function 
\[ \Delta_{\bf m}(x) := \Delta_1(x)^{m_1-m_2} \Delta_2(x)^{m_2-m_3}
\cdot\ldots\cdot \Delta_q(x)^{m_q},\]
and the $\Delta_i(x)$ are the principal minors of the determinant $\Delta(x)$, see
Ch.~XI of \cite{FK} for details. The $\Phi_{\bf m}$ are homogeneous of degree
$|{\bf m}|$ and satisfy $\Phi_{\bf m}(0)=0$ for $\bf m \ne0$,
 $\Phi_{0}(0)=1$, and $\Phi_{\bf m}(I_q)=1$ for the identity matrix $I_q\in\b
C^{q,q}$.

We also consider the renormalized, so-called zonal polynomials
$Z_{\bf m} = c_{\bf m} \Phi_{\bf m}$
with the constants 
\begin{equation}c_{\bf m} := \frac{(q)_{\bf m}|{\bf m}|!}{h({\bf
    m})^2}>0.
\end{equation}
 This normalization is characterized by  
\begin{equation}\label{traceid}
(tr \,x)^k \,=\, \sum_{|{\bf m}|=k} Z_{\bf m}(x)
\quad\quad{\rm for}\>\> k\in \b N_0.
\end{equation}
In fact, the normalization constant $c_{\bf m}$ can be easily derived from
(\ref{traceid}) and some formulas on pp. 237--239 of  \cite{F};
 see also Section XI.5~of \cite{FK} or \cite{Kan}.
 Clearly, we have $Z_{\bf m}(I_q)=c_{\bf m}$.

 By construction, the $ \Phi_{\bf m}$ and  $Z_{\bf m}$
 are invariant under conjugation by $U_q$ and thus depend only on
 the eigenvalues of their argument.
More precisely, for Hermitian $x\in M_{q,q}$ with eigenvalues 
$\xi = (\xi_1, \ldots, \xi_q)\in \b R^q$, we have $Z_{\bf m}(x) = C_{\bf m}^1(\xi)$
for the Jack polynomials  $C_\lambda^1$ (c.f. \cite {FK}, \cite{R1}).
They are homogeneous of degree $|{\bf m}|$ and symmetric in their arguments.

We also introduce the generalized binomial coefficients $\binom{\bf m}{\bf
  n}$ for partitions ${\bf m, n}$ by the unique expansion
$$ \Phi_{\bf m}(I_q+x)=\sum_{|{\bf n}|\le |{\bf m}|} \binom{\bf m}{\bf n}
\Phi_{\bf n}(x)$$
with the identity matrix $I_q\in\b C^{q,q}$. These Binomial coefficients
satisfy  $\binom{\bf m}{\bf
  n}\ne0$ only for ${\bf n}\subset{\bf m}$,
 i.e.~for $n_i\le m_i$ for $i=1,\ldots,q$. Moreover, it follows from \cite{L} that $\binom{\bf m}{\bf
  n}\ge0$, and that for integers $k$,
\begin{equation}\label{binomi-sum}
\sum_{|{\bf n}|=k} \binom{\bf m}{\bf
  n} = \binom{|{\bf m}|}{k}.
\end{equation}
\end{jack}

\begin{Laguerre}\label{facts-laguerre-bessel}

According to p.~343 of \cite{FK} we define the multivariate Laguerre
polynomials
\begin{equation}\label{def-lagu-pol}
L_{\bf m}^p(x):=\sum_{|{\bf n}|\le |{\bf m}|}  \binom{\bf m}{\bf n}
\frac{(p)_{\bf m}}{(p)_{\bf n}}\cdot \Phi_{\bf n}(-x)
\end{equation}
and the associated  multivariate Laguerre functions
\begin{equation}\label{def-lagu-funct}
l_{\bf m}^p(x):= e^{-tr(x)} L_{\bf m}^p(2x)\end{equation}
for $x\in\b C^{q,q}$.
 The functions $L_{\bf m}^p$ and $l_{\bf m}^p$ are also invariant
under conjugation by $U_q$ and may thus be regarded as functions in their
eigenvalues, i.e., as functions on $\Xi_q$. We shall do this from now on
without separate notation.

With a slight change of notation, these Laguerre polynomials and functions are
also considered by Baker and Forrester \cite{BF} in the context of
Calogero-Sutherland models and Dunkl operators. In fact, a comparison of the
notions in \cite{FK} and \cite{BF} shows that our polynomials
$L_{\bf m}^p(x)$ defined above agree with the Laguerre polynomials
$ |{\bf m}|!\cdot L_{\bf m}^{p-q}(x;1)$ in the notion of Proposition 4.3 of \cite{BF}:
\begin{equation}\label{lagu-funct-BF}
L_{\bf m}^p(x) =  |{\bf m}|!\cdot L_{\bf m}^{p-q}(x;1) \quad\quad(\text{in the sense of \cite{BF}}).
\end{equation}
\end{Laguerre}

We next collect some known properties about these  Laguerre polynomials:

\begin{lemma}\label{lagu-ortho}
 The polynomials $L_{\bf m}^p(x)$ form an orthogonal basis on the
  Hilbert space $L^2(\Xi_q, d\mu_{p,q})$ with the measure
$$d\mu_{p,q}(x):=\prod_{i=1}^q(e^{-x_i} x_i^{p-q}) \cdot \prod_{i<j} (x_i-x_j)^2 \> dx.$$
Moreover, for each partition $\bf m$,
$$\int_{\Xi_q} (L_{\bf m}^p(x))^2 \> d\mu_{p,q}(x)= 
d_{p,q}\cdot\frac{|{\bf m}|!(p)_{\bf m}}{q! c_{\bf
    m}}=d_{p,q}\cdot\frac{(p)_{\bf m} h({\bf m})^2}{q!(q)_{\bf m}} $$
with the normalization constant
$$ d_{p,q}:=\mu_{p,q}(\Xi_q)=\int_{\Xi_q}1\> d\mu_{p,q}(x).$$
Finally, $L_{\bf m}^p(0)=(p)_{\bf m}$. 
\end{lemma}

\begin{proof} For the orthogonality and the normalization we refer to
Corollary XV.4.3 of \cite{FK} or Proposition 4.10 of \cite{BF}. As the 
 $L_{\bf m}^p$ form a basis of all polynomials in  $q$ dimensions (use
e.g. Proposition 4.3 of  \cite{BF} and the fact that the Jack polynomials form
a basis), the
completeness of the system  $(L_{\bf m}^p)_{\bf m}$ can be derived by a classical Fourier
argument like 
in the one-dimensional case for Laguerre polynomials. Another possibility here
is to use results of \cite{dJ}.
\end{proof}

We next turn to  multivariate Bessel functions of two arguments $\xi,\eta\in \b
C^q$:

\begin{Bessel}
According to Kaneko \cite{Kan} (see also Section 2.2 of \cite{R1}) we
put
\begin{equation}\label{def-bessel}
J_p(\xi,\eta):= \sum_{\bf m} \frac{(-1)^{|{\bf m}|}}{(p)_{\bf m} |{\bf m}|!} \cdot
\frac{C_{\bf m}^1(\xi)C_{\bf m}^1(\eta)}{C_{\bf m}^1(1,\ldots,1)}.
\end{equation}
For $\eta\in\Xi_q$ we now define the functions $\psi_\eta^p\in C_b(\Xi_q\times\b R)$ by
\begin{equation}\label{def-psi}
\psi_\eta^p(\xi,t):=J_p(\xi^2/2,\eta^2/2)
\end{equation}
according to Section 4.2 of \cite{R1}.We denote the set of all $\psi_{\eta}^p$ with 
 $\eta\in\Xi_q$ by  $\Sigma_{p,q}^2$. 
\end{Bessel}

 The  multivariate Bessel functions appear as limits of the Laguerre functions
 above.
 For the group case with integers $p$,
this was observed already by Faraut \cite{F}.

\begin{lemma}\label{laguerre-to-bessel}
 Let $p\ge 2q-1$, and let $({\bf m}_k)_k$ be a sequence of partitions and
 $(\lambda_k)_k\subset\mathbb R\setminus\{0\}$
 a sequence with $\lambda_k\to0$ and
 $\lim_{k\to\infty} \lambda_k\cdot{\bf m}_k=\eta\in\Xi_q$. Then, for $p\ge 2q-1$,
$$\lim_{k\to\infty}\frac{L_{{\bf
    m}_k}^p(|\lambda_k|\xi^2)}{L_{{\bf m}_k}^p(0)} =J_p(\xi^2/2,\eta^2/2)$$
 uniformly on compact subsets for $\xi\in\Xi_q$.
\end{lemma}

\begin{proof} Writing the expansions of the Laguerre polynomials and
 Bessel functions above in terms of so-called shifted 
Schur functions precisely as on pp. 240--241 of \cite{F}, it can be
 checked as in Proposition 3.3 of \cite{F} that for all partitions $\bf n$, the coefficients  
of the expansion  of $\frac{L_{{\bf
    m}_k}^p(|\lambda_k|\xi^2)}{L_{{\bf m}_k}^p(0)} $ tend to the corresponding
coefficients of $J_p(\xi^2/2,\eta^2/2)$. Moreover, as $\Phi_{\bf n}$ is
homogeneous of degree $|{\bf n}|$ with  $| \Phi_{\bf n}(x)|\le 1$ for
$\|x\|_2=1$, and as 
$$(p)_{\bf n}\ge (q)_{n_1}\cdots  (q)_{n_q}\ge ((q)_{\lfloor|{\bf
    n}|/q\rfloor})^q 
\quad\quad \text{for} \quad p\ge 2q-1,$$
 we obtain with 
 Eq.~(\ref{binomi-sum}) that
\begin{align}
\Biggl|\frac{L_{{\bf
    m}_k}^p(|\lambda_k|\xi^2)}{L_{{\bf m}_k}^p(0)}\Biggr|&\le \sum_{{\bf n}}  \binom{{\bf m}_k}{\bf n} 
\frac{1}{(p)_{\bf n}}\cdot |\Phi_{\bf n}(\xi^2)|\cdot |\lambda_k|^{|\bf n|}
\notag\\
&\le \sum_j \sum_{|{\bf n}|=j}  \binom{{\bf m}_k}{\bf n}
 \frac{1}{((q)_{\lfloor j/q\rfloor})^q }\cdot |\lambda_k|^{j}\|\xi\|^{2j}
\notag\\
&\le \sum_{j} \frac{(|{\bf m}_k||\lambda_k|)^j}{ j!\cdot 
    ((q)_{\lfloor j/q\rfloor})^q }\cdot \|\xi\|^{2j} \quad<\quad\infty.
\notag
\end{align}
locally uniformly for $\xi\in\Xi$. This readily implies the claimed locally
uniform convergence.
\end{proof}

Let us return to the characters of the hypergroups  $(\Xi_q\times\b
R,\circ_{p,q})$.

\begin{definition}\label{character-laguerre-bessel}
For $\lambda\in\b R\setminus\{0\}$ and a partition ${\bf m}$, we define the
function
$\phi_{\lambda,\bf m}^p\in C_b(\Xi_q\times\b R)$ by
\begin{equation}\label{def-phi}
\phi_{\lambda,\bf m}^p(\xi,t):=e^{i\lambda t} \cdot \frac{l_{\bf
    m}^p(|\lambda| \xi^2/2)}{l_{\bf m}^p(0)}
= e^{i\lambda t} \cdot e^{-|\lambda|(\xi_1^2+\ldots~\xi_q^2)/2} \cdot\frac{L_{\bf
    m}^p(|\lambda|\xi^2)}{L_{\bf m}^p(0)}
\end{equation}
with $\xi^2:=(\xi_1^2,\ldots,\xi_q^2)$.
We denote the set of all $\phi_{\lambda,\bf m}^p$ with 
 $\lambda\in\b R\setminus\{0\}$ and  partitions ${\bf m}$ by
$\Sigma_{p,q}^1$.

We  notice that for integers $p$, the functions in the set  $\Sigma_{p,q}^1$ agree with the
spherical functions 
$\phi(\lambda,{\bf m}, .,.)$ of pp.~238--241 of Faraut \cite{F} where there
the Laguerre polynomials are defined with some parameter shift.

Furthermore, we denote the set of all Bessel functions $\psi_{\eta}^p$ with 
 $\eta\in\Xi_q$ by  $\Sigma_{p,q}^2$. Again, for integers $p\ge1$,
the set  $\Sigma_{p,q}^2$ consists of spherical functions by \cite{F}.
We note that Faraut (p.~241 of \cite{F}) uses a slightly different
notion for these Bessel functions; in his notion we have
$$\psi(\eta,\xi)=J_p(\xi^2,\eta)=\psi_{\sqrt{2\eta}}^p(\sqrt 2 \cdot \xi,t)
\quad\quad(t\in\b R \quad \text{arbitrary}).$$
\end{definition}

\begin{theorem} Let $p\ge 2q-1$. Then all functions in $\Sigma_{p,q}:=\Sigma_{p,q}^1\cup
  \Sigma_{p,q}^2 $ are characters of  the hypergroup $(\Xi_q\times\b R,\circ_{p,q})$.
\end{theorem}

\begin{proof}
By Corollary \ref{mult}, it suffices to show that all functions in
$\Sigma_{p,q}$ are multiplicative and bounded.
Taking Remark \ref{spezial2}(3) and the results of Section 4 of \cite{R1} into
account, this is clear for all functions in  $\Sigma_{p,q}^2$.

The proof for the Laguerre functions in $\Sigma_{p,q}^1$ is slightly 
more involved. Here we first consider the
group cases with integers $p\ge 2q-1$ where  $(\Xi_q\times\b R,\circ_{p,q})$
is  isomorphic with the orbit hypergroup which appears when the group
$U_p\times U_q$ acts on the Heisenberg group $H_{p,q}$.
In this cases it is well known that the functions in $\Sigma_{p,q}^1$
correspond to bounded spherical functions on the associated Gelfand pairs; see
\cite{F} and references cited there.
 For general parameters $p>2q-1$ we again
employ analytic continuation by Carlson's theorem \ref{continuation}. For this
we fix a partition $\bf m$, $\lambda\in\b R^\times$, and $(\xi,a),(\eta,b)\in
\Xi_q\times \b R$, and consider the function
\begin{align}
F(p):= & \phi_{\lambda,\bf m}^p(\xi,a)\cdot \phi_{\lambda,\bf
  m}^p(\eta,b)\notag \\ 
& -\kappa_{p,q}\int_{B_q}\int_{U_q}
 \phi_{\lambda,\bf m}^p
\Bigl(\sigma(\sqrt{\xi^2 + u\eta^2 u^* +
 \xi w u\eta u^*+  u\eta u^*w^*\xi}),
\notag \\ 
&\quad\quad\quad\quad\quad\quad\quad\quad
a+b- {\rm Im\>} tr(\xi w u\eta u^*)\Bigr)
\cdot
\Delta(I_q-w^*w)^{p-2q}\, du\, dw
\notag
\end{align}
which has zeros for integer values $p\ge 2q$. Moreover, $F$ is analytic on
$W:=\{p\in\b C:\> {\rm Re}\> p>2q-1\}$. Furthermore, because of
$$|(p)_{\bf m}|=\bigl|\prod_{j=1}^q (p-j+1)_{m_j}\bigr|\ge 1$$
for $p\in W$ and the definition of the Laguerre functions
$\phi_{\lambda,\bf m}^p$ above, we see that $\phi_{\lambda,\bf m}^p$ remains bounded
for $p\in W$ locally uniformly for $(\xi,a)\in
\Xi_q\times \b R$. Therefore, for a suitable constant $C>0$,
$$|F(p)|\le C^2 + C\cdot
|\kappa_{p,q}|\int_{B_q}\int_{U_q}|\Delta(I_q-w^*w)^{p-2q}|\, du\, dw 
\le C^2 + C^2\cdot
|\kappa_{p,q}| $$
with $|\kappa_{p,q}|=O(|p|^{q^2})$ by (\ref{kappaestim}) for  $p\in W$.
 Therefore, by Theorem \ref{continuation}, $F(p)=0$ for all
$p> 2q-1$ which proves that all elements of   $\Sigma_{p,q}^1$ are
multiplicative for all $p\ge 2q-1$. We finally note that the Laguerre
functions in  $\Sigma_{p,q}^1$
are obviously bounded by their definition.
\end{proof}

As bounded multiplicative functions $\alpha$ on a commutative hypergroup satisfy
$\|\alpha\|_\infty=1$, we obtain:

\begin{corollary}
For all $p\ge 2q-1$, $\lambda\in\b R\setminus\{0\}$ and all partitions ${\bf
  m}$, $\|\phi_{\lambda,\bf m}^p\|_\infty=1.$
\end{corollary}

We next turn to the Plancherel measures $\pi_{p,q}$ of the hypergroups
$(\Xi_q\times\b R,\circ_{p,q})$.
 These measures are well-known for the group cases with integers $p$ by
 \cite{BJRW}, \cite{BJR2},  \cite{F}, as well as for $q=1$ in the general
 case; see e.g. Section 8.1 of \cite {T}.

\begin{theorem}\label{Plancherel} Let $p\ge 2q-1$. If the hypergroup 
$(\Xi_q\times\b R,\circ_{p,q})$ is equipped with the Haar measure $\tilde\omega_{p,q}$
according to Theorem \ref{main3},   then the associated
  Plancherel measure $\pi_{p,q}\in M^+(\Sigma_{p,q})$
 according to Section \ref{facts-comm-hyper}
 is given by
$$\pi_{p,q}(g) =\frac{2^q \cdot q!}{d_{p,q}\cdot 2\pi}\cdot \int_{\b R\setminus\{0\}}\sum_{\bf m}
\frac{(p)_{\bf m}(q)_{\bf m}}{h({\bf m})^2}\cdot
g({\bf m},\lambda)\> |\lambda|^{pq} d\lambda$$
for  $g\in L^1(\Sigma_{p,q},\pi_{p,q})$ where clearly the set $\Sigma_{p,q}$
of functions is identified with its describing parameter set in the obvious way.
\end{theorem}

\begin{proof} Let $f\in C_c(\Xi_q\times\b R)$. The classical Plancherel
  formula for $(\b R, +)$ says that the classical Fourier transform 
$$F(x,\lambda):= \int_{ \b R} f(x,t)\> e^{-i\lambda t} \> dt$$
w.r.t.~the variable $t$ satisfies
$$\|f\|_{2, \tilde\omega_{p,q}}^2= \frac{1}{2\pi} \cdot
\int_{\Xi_q }\Bigl(\int_{ \b R}|F(x,\lambda)|^2\>
d\lambda\Bigr) h_{p,q}(x)\> dx.$$
For fixed $\lambda\in \b R\setminus\{0\}$, we now consider the
renormalized Laguerre functions 
$$\tilde\phi_{|\lambda|,\bf m}^p(z):= l_{\bf m}^{p}(|\lambda| z^2/2)/ l_{\bf m}^{p}(0)$$
which satisfy $\phi_{\lambda,\bf m}^p(z,t)=e^{i\lambda
  t}\tilde\phi_{|\lambda|,\bf m}^p(z)$. We then obtain with the Haar density
(\ref{haar-density}), the transformation formula, and
$h_{p,q}(cx)=c^{2pq-q}\cdot h_{p,q}(x)$ for $c>0$ that
\begin{equation}\label{renorming-lagu}
c_{\lambda,\bf m}:=\int_{\Xi_q }|\tilde\phi_{|\lambda|,\bf m}^p(x)|^2
h_{p,q}(x)\> dx
=\frac{1}{ |\lambda|^{pq}}
\int_{\Xi_q }\frac{l_{\bf m}^{p}( z^2/2)^2} {l_{\bf m}^{p}(0)^2}h_{p,q}(z)\> dz. 
\end{equation}
Moreover, according to Lemma \ref{lagu-ortho},
the functions $\tilde\phi_{1,\bf m}^p/\sqrt{ c_{1,\bf
    m}}$ form an orthonormal basis of  $L^2( \Xi_q, h_{p,q}(x)\> dx)$. Therefore, by
Parseval's identity,
\begin{equation}
\|f\|_{2, \tilde\omega_{p,q}}^2=\frac{1}{2\pi}
 \sum_{\bf m\ge 0}\int_{ \b R}\Biggl|\int_{\Xi_q }
F(x,\lambda)
\overline{\tilde\phi_{|\lambda|,\bf m}^p(x)}\cdot h_{p,q}(x)\> dx\Biggr|^2
\cdot \frac{1}{ c_{\lambda,\bf  m}} \> d\lambda
\end{equation}
As by Section  \ref{facts-comm-hyper} and Theorem  \ref{main3} 
the hypergroup Fourier transform $\hat f$ is given by
\begin{align}
\hat f(\lambda,\bf m)&=\int_{\Xi_q}\int_{ \b R}\overline{\phi_{|\lambda|,\bf
    m}^p(x)}\cdot f(x,t)\> dt \> h_{p,q}(x)\> dx
\notag\\
&=
\int_{\Xi_q} F(x,\lambda)\tilde\phi_{|\lambda|,\bf m}^p(x)\cdot h_{p,q}(x)\> dx,
\notag
\end{align}
we obtain from Lemma \ref{lagu-ortho} that
\begin{align}
\|f\|_{2, \tilde\omega_{p,q}}^2&=\frac{1}{2\pi}
 \sum_{\bf m\ge 0}\int_{ \b R} |\hat f(\lambda,{\bf m})|^2\cdot
\frac{1}{ c_{\lambda,\bf  m}} \> d\lambda
\notag\\
&=\frac{2^q \cdot q!}{d_{p,q}\cdot 2\pi}\cdot \sum_{\bf m\ge 0}\int_{ \b R} 
|\hat f(\lambda,{\bf m})|^2\cdot |\lambda|^{pq}\> d\lambda \cdot 
\frac{(p)_{\bf m}(q)_{\bf m}}{h({\bf m})^2}
\notag\\&= \|\hat f\|_{2,\pi_{p,q}}^2
\notag
\end{align}
with the measure $\pi_{p,q}$ introduced in the theorem.
As  $C_c(\Xi_q\times\b R)$ is dense in $L^2(\Xi_q\times\b R, h_{p,q}(x)\> dx\>
dt)$, the equation holds also for all $f\in L^2(\Xi_q\times\b R, h_{p,q}(x)\> dx\>
dt)$, which caracterizes the Plancherel measure as claimed.
\end{proof}

We next prove  that the characters in  $\Sigma_{p,q}$ form the complete dual
space $\hat X$ of $(\Xi_q\times\b R, \circ_{p,q})$. The following
representation of the topology on  $\hat X$ as a Heisenberg fan 
is due to J.~Faraut who considered the group case with integer $p$:

\begin{theorem}\label{complete-dual}
 Let $p\ge 2q-1$. Then the set   $\Sigma_{p,q}$ is equal to the
  complete dual space $\hat X$ of the hypergroup
  $(\Xi_q\times\b R, \circ_{p,q})$. More precisely,
if the closed subset 
$$D:=(\{0\}\times\Xi_q)\cup
 \{(\lambda,\lambda\cdot{\bf m}):\> \lambda\in\b R\setminus\{0\},\>{\bf
   m}\text{ a partition}\}$$ 
of $\b R^{q+1} $
carries the usual topology as a subset of $\b R^{q+1} $, and if
  $\hat X=\Sigma_{p,q}$
  carries the topology of locally uniform convergence, then the mapping
  $$E:\Sigma_{p,q}\to D$$
 with
$$\phi_{\lambda,\bf m}^p\in \Sigma_{p,q}^1\mapsto (\lambda,\lambda
  m_1,\ldots,\lambda m_q)$$
and
$$\psi_\rho^p\in \Sigma_{p,q}^2\mapsto (0,\rho_1,\ldots,\rho_q)$$
establishes an homeomorphism.
\end{theorem}

\begin{proof} We first notice that it follows easily from the definition of the
 functions $\phi_{\lambda,\bf m}^p$ and $\psi_\rho^p$ that these functions
on the hypergroup $X:=\Xi_q\times\b R$ are different for different indices.

We now prove that   $\Sigma_{p,q}$ is equal to the
  complete dual space $\hat X$.
For this consider the subgroup $G:=\{0\}\times\b R$ of
 $X$ as above and  the associated annihilator
$$A(\hat X,G):=\{\alpha\in\hat X:\> \alpha|_G\equiv 1\}.$$
This set is a closed subset of the dual $\hat X$ and can by \cite{V2} be identified with the
dual space of the quotient hypergroup  $(\Xi_q\times\b R)/G$ which was studied 
in Remark \ref{spezial2}(3). In fact, this hypergroup $(\Xi_q\times\b R)/G$ is (isomorphic to)
 a Bessel-type hypergroup on the Weyl chamber $\Xi_q$ as studied in section 4
 of \cite{R1}. We thus conclude readily from  Section 4 of \cite{R1} and the
 general results on annihilators in \cite{V2} 
that $A(\hat X,G)=  \Sigma_{p,q}^2$ holds,  and that the mapping $E$ above restricted
to  $\Sigma_{p,q}^2$ is an homeomorphism.

We next turn to the complete dual $\hat X$. Here we first conclude from
Theorem \ref{Plancherel} and Corollary \ref{mult} that the dual $\hat X$ 
is the closure of $\Sigma_{p,q}^1$ in $C_b(\Xi_q\times\b R)$
w.r.t.~the topology of locally uniform convergence. To prove that this
closure is equal to   $\Sigma_{p,q}$, we
consider some
$\alpha\in\hat X\subset C_b(\Xi_q\times\b R)$ 
which is the locally uniform limit of a sequence 
$(\phi_{\lambda_n,\bf m_n}^p)_{n\ge1}\subset \Sigma_{p,q}^1$. As the
restrictions to $G$ also converge locally uniformly, it follows from the
definition of the  $\phi_{\lambda,\bf m}^p$ that $(\lambda_n)_{n\ge1}\subset\b
R\setminus\{0\}$
converges to some $\lambda\in\b R$. We consider two cases:
\begin{enumerate}
\item[\rm{(1)}] If $\lambda=0$, then it follows from the definition of the 
 $\phi_{\lambda,\bf m}^p$ that $\alpha\in A(\hat X,G)$, and the preceding
  consideration implies that  $\alpha\in\Sigma_{p,q}^2$.
\item[\rm{(2)}] Let $\lambda\ne 0$. If $({\bf m}_n)_{n\ge1}$ remains bounded, we
   may choose it as a constant sequence with ${\bf m}_n={\bf m}$
 without loss of
   generality,
and we obtain from 
$\lambda_n\to\lambda$ that $\alpha=\lim_n\phi_{\lambda_n,\bf m}^p=
\phi_{\lambda,\bf m}^p\in\Sigma_{p,q}^1$. We thus may restrict our attention
to the case where at least one component of ${\bf m}_n$ tends to infinity.
 In this case we take an arbitrary
  multiindex ${\bf m}$ and conclude from the  convergence of
  normalized Laguerre polynomials 
$L_{\lambda_n,\bf m}^p$ to $L_{\lambda,\bf m}^p$ (see Section
  \ref{facts-laguerre-bessel}) that for the density $h_{p,q}$ of (\ref{haar-density})
\begin{align}
\int_{\Xi_q} &\alpha(z,0)\cdot \overline{\phi_{\lambda,\bf m}^p(z,0)}
h_{p,q}(z)\> dz 
\notag\\&
=
\int_{\Xi_q}  \lim_{n\to\infty}\bigl(\phi_{\lambda_n,{\bf m}_n}^p(z,0)\cdot
 \overline{\phi_{\lambda_n,\bf m}^p(z,0)}\bigr)h_{p,q}(z)\> dz\notag\\
&=\int_{\Xi_q}  \lim_{n\to\infty}\bigl(\frac{L_{{\bf m}_n}^p(|\lambda_n|z^2/2)
L_{\bf m}^p(|\lambda_n|z^2/2)}{L_{{\bf m}_n}^p(0)L_{\bf m}^p(0)}
\cdot e^{-|\lambda_n|\|z\|^2}\bigr)h_{p,q}(z)\> dz.
\notag
\end{align}
Using a renormalization
 as in Eq.~(\ref{renorming-lagu}) with $\lambda_n\to\lambda$,
 the dominated convergence theorem, and the
 fact that the modulus of the $\phi_{\lambda_n,{\bf m}_n}^p$ is bounded by 1,
 we obtain that that this expression is equal to
$$
\frac{1}{|\lambda|^{pq}} \lim_{n\to\infty}\int_{\Xi_q}
\frac{L_{{\bf m}_n}^p(z^2/2)
L_{\bf m}^p(z^2/2)}{L_{{\bf m}_n}^p(0)L_{\bf m}^p(0)}
\cdot e^{-\|z\|^2} h_{p,q}(z)\> dz  =0.$$
Therefore, as $\alpha\in C_b(\Xi_q\times\{0\})\subset L^2(\Xi_q,
e^{-\lambda\|z\|^2} h_{p,q}(z)dz)$, and as the Laguerre polynomials $(L_{\lambda,\bf
  m}^p(z^2))_{\bf m}$ with squared arguments form an orthogonal basis of this
$L^2$-space
 by Lemma \ref{lagu-ortho},
 it follows that $\alpha=0$ a.s. on
$\Xi_q\times\{0\}$ and thus on $\Xi_q\times\b R$. This is a contradiction of
$\alpha$ being a continuous character with $\alpha(0,0)=1$.
\end{enumerate}
 Summarizing, we conclude  that 
 $\Sigma_{p,q}=\hat X$, and that a sequence
 $(\phi_{\lambda_n,{\bf m}_n}^p)_{n\ge1}\subset \Sigma_{p,q}^1$ can converge
 to
 a character without loss of generality  only in the following two cases:
 Either $\lambda_n\to\lambda \ne0$ and $({\bf m}_n)_n$
is finally constant or $\lambda_n\to0$. In the first case,
 locally uniform convergence obviously appears, and in the second case
we have locally uniform  convergence for $\lambda_n{\bf m}_n\to \eta\in \Xi_q$
by Lemma \ref{laguerre-to-bessel}.

We finally prove that a
sequence $(\phi_{\lambda_n,{\bf m}_n}^p)_{n\ge1}\subset \Sigma_{p,q}^1$ can
converge locally uniformly to some $\psi_\rho^p$ {\bf only} for parameters with $\lim_{n}
\lambda_n{\bf m}_n=\rho$.
If this is done, it follows readily from our preceding informations that the
mapping $E$ of our theorem is an homeomorphism. In order to prove this
statement, consider such a sequence $(\phi_{\lambda_n,{\bf m}_n}^p)_{n\ge1}$
and its limit $\psi_\rho^p$. If the sequence $(\lambda_n{\bf
  m}_n)_n\subset\Xi_q$ is bounded, we find a convergent subsequence with some
limit $\tilde\rho\in\Xi_q$. For this subsequence we obtain 
$\phi_{\lambda_{n_k},{\bf m}_{n_k}}^p\to\psi_{\tilde\rho}^p$
locally uniformly, and thus $\psi_{\tilde\rho}^p=\psi_\rho^p$, i.e., $\tilde\rho=\rho$.
Therefore, each convergent subsequence of the bounded sequence $(\lambda_n{\bf
  m}_n)_n\subset\Xi_q$ converges to $\rho$ which implies $\lambda_n{\bf
  m}_n\to\rho$ as claimed.
We finally consider the unbounded case. Here we may assume without loss of
generality that the largest component of $(\lambda_n{\bf
  m}_n)_n\subset\Xi_q$ converges to $\infty$, i.e., that $\lambda_n{\bf
  m}_{1,n}\to\infty$. We then define $\tilde\lambda_n:=\rho/{\bf  m}_{1,n}$
and observe that $\tilde\lambda_n{\bf m}_n$ has a subsequence which
converges to some $\tilde\rho\in\Xi_q$ with $\tilde\rho_1=1$.
For this subsequence and any $(\xi,t)\in\Xi_q\times\b R$ we the have
$$\phi_{\tilde \lambda_{n_k},{\bf m}_{n_k}}^p(\xi,t)\to
\psi_{\tilde\rho}^p(\xi,t)$$
as well as by the definition of  $\phi_{\lambda,{\bf m}}^p$ and our assumption
about locally uniform convergence,
$$\phi_{\tilde \lambda_{n_k},{\bf m}_{n_k}}^p(\xi,t)=\phi_{ \lambda_{n_k},{\bf
    m}_{n_k}}^p
(\xi\cdot\sqrt{\tilde \lambda_{n_k}/\lambda_{n_k}},t\cdot\tilde \lambda_{n_k}/\lambda_{n_k})
\to
\psi_{\rho}^p(0,0)=1.$$ It follows that $\psi_{\tilde\rho}^p\equiv 1$
contradicting $\tilde\rho_1=1$. Therefore this limit case cannot appear which
completes the proof of the theorem.
\end{proof}

The proof of the continuity of the mapping $E$ above is quite special. It was
pointed out to the author by J.~Faraut that a more systematic approach for this
part of the theorem is available by using heat kernels on the hypergroups
$\Xi_q\times\b R$ and their explicit  continuous hypergroup Fourier transfoms.

 We finally turn to a refinement of the last part of the theorem.
The following observation is clear from Section 2:

\begin{lemma} Let $(\xi,a),(\eta,b)\in \Xi_q\times\b R$ with $\xi_q>0$ and
  $\eta_q>0$. Then, for $p>2q-1$, the convolution product 
$\delta_{(\xi,a)} \circ_{p,q} \delta_{(\eta,b)}$ is absolutely continuous
  w.r.t.~the Haar measure $\tilde\omega_{p,q}$.
\end{lemma}

\begin{proof} As absolute continuity is preserved under the continuous
  projection $\Phi:\Pi_q\times\b R\to \Xi_q\times\b R$, the lemma follows
  immediately from Remark \ref{abs-cont}.
\end{proof}

As by the Riemann-Lebesgue Lemma for hypergroups (see \cite{J})
 the hypergroup Fourier transform maps functions in $L^1(\Xi_q\times\b
R,\tilde\omega_{p,q})$ to $C_0$-functions on the dual space $\hat X$, the preceding lemma
and Theorem \ref{complete-dual} imply the following.

\begin{corollary} Let $(\xi,a)\in \Xi_q\times\b R$ with $\xi_q>0$.  Then, for
  $p>2q-1$, the hypergroup Fourier transform of $\delta_{(\xi,a)}$ on the dual
  space $\hat X$ is a $C_0$-function. In particular, for $\xi\in \Xi_q$ with  $\xi_q>0$,
$$\lim_{\lambda\cdot m_1\to\infty} \frac{l_{\bf
    m}^p(\lambda \xi^2/2)}{l_{\bf m}^p(0)}=0.$$
 
\end{corollary}

\section{A product formula for Laguerre functions}

In this section we derive a product formula for the multivariate Laguerre functions $l_{\bf m}^p$ of
Section \ref{facts-laguerre-bessel} for $p\ge 2q-1$. For $q=1$, this 
formula was established directly by Koornwinder \cite{Ko} 
who also discusses its connection with Heisenberg groups.
We here  derive the product formula from the product 
formula (\ref{convo-w2}) for $p>2q-1$ and its degenerate version for
 $p=2q-1$ according to Remark \ref{spezial2}(1)
for the characters $\phi_{\lambda,\bf m}^p\in\Sigma_{p,q}^1$ of the
 commutative hypergroups  $(\Xi_q\times\b R,\circ_{p,q})$.
We here shall use the general approach of \cite{RV} where, embedded
 into a more general setting, it is also explained how for
 $q=1$ Koornwinder's product formula for  the one-dimensional Laguerre functions corresponds
 to the Heisenberg-type hypergroup convolution on $[0,\infty[\times \b R$.
 We now extend this approach from $q=1$ to $q\ge 1$.

For this we recapitulate from Remark \ref{spezial2} that $G:=\{0\}\times \b R$
 is a subgroup of the commutative hypergroup 
 $(\Xi_q\times\b R,\circ_{p,q})$ for $p\ge 2q-1$. Moreover, $\tau(x,t):=e^{it}$
defines a function $\tau\in C_b( \Xi_q\times\b R)$ with
$$|\tau(x,t)|=1, \quad \tau(\overline{(x,t)})=\overline{\tau(x,t)}, \quad\text{and}\quad
 \tau((x,t)\cdot(0,s))=\tau(x,t)\cdot\tau(0,s)$$
for all $x\in\Xi_q$ and $s,t\in \b R$. In other words, $\tau$
 is a partial character of $(\Xi_q\times\b R,\circ_{p,q})$ 
with respect to $G$ according to Definition 4.1 of  \cite{RV}.
We now consider the canonical projection 
$$p:\Xi_q\times\b R\to (\Xi_q\times\b R)/G \simeq \Xi_q$$
according to Remark \ref{spezial2}(3) and recapitulate that
 the quotient hypergroup $(\Xi_q\times\b R)/G$ 
agrees with the corresponding Hermitian Bessel hypergroup
 on $\Xi_q$ of  R\"osler  \cite{R1}, i.e., the corresponding hypergroup involution 
is the identity mapping. Following Section 4 of  \cite{RV},
 we now define a deformed quotient convolution of point measures  on  $\Xi_q$ 
by
\begin{equation}\label{lagu-convo}
\delta_\xi\bullet_{\tau,p,q}\delta_\eta := p(\tau\cdot(\delta_{(\xi,0)}\circ_{p,q}\delta_{(\eta,0)})) 
\quad\quad\text{for}\quad \xi,\eta\in\Xi_q.
\end{equation}
According to  Section 4 of  \cite{RV}, this convolution can be uniquely 
extended in a weakly continuous bilinear way to a commutative Banach-$*$-algebra
$(M_b(\Xi_q),\bullet_{\tau,p,q})$ with the total variation norm as norm. More precisely, by Theorem 4.6 and 
Corollary 4.7 of  \cite{RV}, $(\Xi_q,\bullet_{\tau,p,q})$
 becomes a Hermitian signed hypergroup in the sense of \cite{R0};
 see also \cite{RV} and \cite{Ross} for the notion of 
 signed hypergroups. 
 
Let us compute the convolution (\ref{lagu-convo}) in an explicit way. Eq.~(\ref{convo-w2}) for 
  $p>2q-1$ shows that for $f\in C_b(\Xi_q)$ we have
\begin{align}\label{lagu-convo2}
\delta_\xi\bullet_{\tau,p,q}\delta_\eta(f)&= \int_{Xi_q} f\> dp(\tau\cdot(\delta_{(\xi,0)}\circ_{p,q}\delta_{(\eta,0)})) 
\notag\\
&= \int_{\Xi_q\times\b R} f(x)\cdot e^{it}\>d(\delta_{(\xi,0)}\circ_{p,q}\delta_{(\eta,0)})(x,t)
\notag\\
&= \kappa_{p,q}\int_{B_q}\int_{U_q}
f\Bigl(\sigma(\sqrt{\xi^2 + u\eta^2 u^* + \xi w u\eta u^*+  u\eta
  u^*w^*\xi})\Bigr)\notag \\ 
&\quad\quad\quad\quad\quad\quad
\cdot e^{-i\cdot {\rm Im\>} tr(\xi w u\eta u^*)}
\cdot
\Delta(I_q-w^*w)^{p-2q}\, du\, dw.
\end{align}
For  $p=2q-1$ one obtains a corresponding degenerate version
 of this formula by using Remark \ref{spezial2}(1) 
and Section \ref{spezail1}.
We notice that this convolution is obviously not probability 
preserving and usually even not positivity preserving.

 Moreover it follows from Theorem 5.2 of  \cite{RV} that
 for $p\ge 2q-1$ and all partitions ${\bf m}$,
 the normalized Laguerre functions
$$\tilde\phi_{\bf m}^p(x):=
 \frac{l_{\bf m}^p(x^2/2)}{l_{\bf m}^p(0)}= e^{-(x_1^2+\ldots+x_q^2)/2} \frac{L_{\bf m}^p(x^2)}{L_{\bf m}^p(0)}
\quad\quad (x\in\Xi_q)$$
 form all bounded $\b R$-valued multiplicative functions on  $(\Xi_q,\bullet_{\tau,p,q})$.
In summary, we have the following product formula:

\begin{corollary} For all  $p> 2q-1$, $\xi,\eta\in\Xi_q$,  and all partitions $\bf m$,
\begin{align}\label{lagu-convo3}
\tilde\phi_{\bf m}^p(\xi)\cdot& \tilde\phi_{\bf m}^p(\eta)= 
\notag \\ &=
\kappa_{p,q}\int_{B_q}\int_{U_q}
\tilde\phi_{\bf m}^p\Bigl(\sigma(\sqrt{\xi^2 + u\eta^2 u^* + \xi w u\eta u^*+  u\eta u^*w^*\xi})\Bigr)
\notag \\ &\quad\quad\quad \quad\quad \quad
\cdot e^{-i\cdot {\rm Im\>} tr(\xi w u\eta u^*)}
\Delta(I_q-w^*w)^{p-2q}\, du\, dw.
\end{align}
Moreover, for  $p= 2q-1$,
\begin{align}\label{convo4}
\tilde\phi_{\bf m}^{2q-1}(\xi)&\cdot \tilde\phi_{\bf m}^{2q-1}(\eta)=
\notag\\ 
&=
\kappa_{2q-1,q}\int_{B^{q-1}}\int_S\int_{U_q}
\tilde\phi_{\bf m}^{2q-1}\Bigl(\sigma(\sqrt{\xi^2 + \eta^2 + \xi P(y)\eta + \xi P(y)^*\eta})\Bigr)\notag\\ 
&\quad\quad
\cdot e^{-i\cdot {\rm Im\>} tr(\xi P(y)  u\eta u^*)}
\cdot\prod_{j=1}^{q-1}(1-\|y_j\|_2^2)^{p-q-j} \>  du\> dy_1\ldots
dy_{q-1}\> ds(y_q)\notag
\end{align}
with $y:=(y_1,\ldots,y_q)$ and $B:=\{y\in\b C^q:\> \|y\|_2<1\}$, 
 where $s\in M^1(S)$ denotes the uniform distribution on the sphere $S:=\{y\in\b C^q:\> \|y\|_2=1\}$, 
and 
 $P$ is the map defined in 
(\ref{trafo-P}).
\end{corollary}

As the convolution $\bullet_{\tau,p,q}$ of (\ref{lagu-convo2}) is not
 mass preserving, there does not exist 
 an associated translation invariant measure $m_{\tau,p,q}\in M^+(\Xi_q)$
 in the usual hypergroup sense. However, 
Theorem 4.6 of \cite{RV} ensures that the Haar measure 
$$dm_{p,q}(\xi)=\prod_{i=1}^q \xi_i^{2p-2q + 1} \prod_{i<j} (\xi_i^2-\xi_j^2)^2 \> d\xi \> \in M^+(\Xi_q)$$
 of the usual quotient hypergroup $(\Xi_q\times\b R)/G \simeq \Xi_q$ 
also admits the following adjoint relation for $\bullet_{\tau,p,q}$:
If for $f\in C_c(\Xi_q)$ and $\xi\in \Xi_q$ we define the translate
$$T_\xi f(\eta):= \int_{\Xi_q} f \> d(\delta_\xi\bullet_{\tau,p,q}\delta_\eta),$$
then we have for all  $f,g\in C_c(\Xi_q)$
\begin{equation}
\int_{\Xi_q} (T_\xi f)\cdot g  \>dm_{p,q}=\int_{\Xi_q} (T_\xi g)\cdot f  \>dm_{p,q}.
\end{equation}

\begin{remark}
Consider the group case with integer parameters $p\ge2q-1$.
 Here, all characters of the double coset hypergroup $(\Xi_q\times\b
 R,\circ_{p,q})$
 correspond to positive definite spherical 
functions on the Heisenberg group $H_{p,q}$ and admit therefore a dual
 product formula, i.e., for all $\alpha,\beta\in \Sigma_{p,q}$ 
(see Theorem \ref{complete-dual}) there is a unique probability
 measure $\rho_{\alpha,\beta}\in M^1(\Sigma_{p,q})$ such that for all
$(x,t)\in \Xi_q\times\b R$,
\begin{equation}\label{allg-dual-prod}
\alpha(x,t)\cdot \beta(x,t) = \int_{\Sigma_{p,q}} \gamma(x,t)\> d\rho_{\alpha,\beta}(\gamma).
\end{equation}
Let us  take $\alpha(x,t):=e^{\lambda_1 it} \cdot \frac{l_{\bf m}^p(|\lambda_1|x^2/2)}{l_{\bf m}^p(0)}$ and
 $\beta(x,t):=e^{\lambda_2it} \cdot \frac{l_{\bf n}^p(|\lambda_2|x^2/2)}{l_{\bf n}^p(0)}$
for $\lambda_1,\lambda_2\ne0$ and partitions $\bf m,n$.

If we consider the $t$-dependence of the product in (\ref{allg-dual-prod})
 for $x=0$, we obtain for $\lambda_1,\lambda_2>0$ that
\begin{align}
\frac{L_{\bf m}^p(\lambda_1x^2/2)}{L_{\bf m}^p(0)}& \cdot
 \frac{L_{\bf n}^p(\lambda_2x^2/2)}{L_{\bf n}^p(0)}= 
\notag\\&=
\sum_{\bf |k|\le |m|+|n| }   c({\bf m, n, k};\lambda_1,\lambda_2, p,q)\cdot 
 \frac{L_{\bf k}^p((\lambda_1+\lambda_2)x^2/2)}{L_{\bf k}^p(0)}\notag
\end{align}
for unique coefficients   $c({\bf m, n, k};\lambda_1,\lambda_2, p,q)$, which satisfy
$$c({\bf m, n, k};\lambda_1,\lambda_2, p,q)\ge0 \quad\quad \text{and} \quad\quad 
\sum_{\bf k} c({\bf m, n, k};\lambda_1,\lambda_2, p,q)=1.$$
For instance, for $\bf n=0$, we get for $0\le\lambda_1\le\lambda_2$ that
$$L_{\bf m}^p(\lambda_1x^2/2)=\sum_{\bf k} c({\bf m,  k}; p,q) L_{\bf k}^p(\lambda_2x^2/2)$$
 with nonnegative coefficients $c({\bf m,  k}; p,q)$.

On the other hand, for $\lambda_1>0$ and $-\lambda_1<\lambda_2<0$ we obtain 
$$\frac{l_{\bf m}^p(\lambda_1x^2/2)}{l_{\bf m}^p(0)} \cdot
 \frac{l_{\bf n}^p(|\lambda_2|x^2/2)}{l_{\bf n}^p(0)}= 
\sum_{\bf k }   c({\bf m, n, k};\lambda_1,\lambda_2, p,q)\cdot
  \frac{l_{\bf k}^p((\lambda_1+\lambda_2)x^2/2)}{l_{\bf k}^p(0)}.$$
We expect that these and further related results also hold for arbitrary
(noninteger)
 $p\ge 2q-1$. For $q=1$, the formulas above are 
connected with the 
discrete Laguerre convolution  in \cite{AG}.
\end{remark}

\end{document}